\documentclass[a4paper]{article}

\usepackage[english]{babel}
\usepackage[T1]{fontenc}
\usepackage[top=1.2 in, bottom = 1.2 in, left = 1.2 in, right = 1.2 in]{geometry}
\usepackage{csquotes}
\usepackage{amssymb,amsmath,amsthm,amsfonts}
\usepackage{mathtools}
\usepackage{MnSymbol}
\usepackage{booktabs}
\usepackage{tikz-cd} 
\usepackage{graphicx}
\usepackage[colorinlistoftodos]{todonotes}
\usepackage{hyperref}
\allowdisplaybreaks
\newtheorem{theorem}{Theorem}
\newtheorem{definition}{Definition}
\newtheorem{proposition}{Proposition}
\newtheorem{conjecture}{Conjecture}
\newtheorem{corollary}{Corollary}
\newtheorem{lemma}{Lemma}
\newtheorem{case}{Case}

\numberwithin{subcase}{case}

\title{The Waldspurger Transform of Permutations and Alternating Sign Matrices}
\author{James McKeown\thanks{I would like to thank my advisor, Drew Armstrong, for introducing me to Waldspurger's result, and for all of his help and insight.}}
\begin{document}
\maketitle

\begin{abstract}
In 2005 J.~L.~Waldspurger proved the following theorem: given a finite real reflection group $W$, the closed positive root cone is tiled by the images of the open weight cone under the action of the linear transformations $id-w$. Shortly thereafter E. Meinrenken extended the result to affine Weyl groups. P.V. Bibikov and V.S. Zhgoon then gave a uniform proof for a discrete reflection group acting on a simply-connected space of constant curvature.

 In this paper we show that the Waldspurger and Meinrenken theorems of type A give a new perspective on the combinatorics of the symmetric group. In particular, for each permutation matrix $w \in \mathfrak{S}_n$ we define a non-negative integer matrix $\mathbf{WT}(w)$, called the {\em Waldspurger transform} of $w$. The definition of the matrix $\mathbf{WT}(w)$ is purely combinatorial but its columns are the images of the fundamental weights under the action of $id-w$, expressed in simple root coordinates. The possible columns of $\mathbf{WT}(w)$ (which we call {\em UM vectors}) are in bijection with many interesting structures including: unimodal Motzkin paths, abelian ideals in the Lie algebra $\mathfrak{sl}_n(\mathbb{C})$, Young diagrams with maximum hook length $n$, and integer points inside a certain polytope.
 
We show that the sum of the entries of $\mathbf{WT}(w)$ is equal to half the entropy of the corresponding permutation $w$, which is known to equal the rank of $w$ in the Dedekind-MacNeille completion of the Bruhat order. Inspired by this, we extend the Waldpurger transform $\mathbf{WT}(M)$ to alternating sign matrices $M$ and give an intrinsic characterization of the image. This provides a geometric realization of Dedekind-MacNeille completion of the Bruhat order (a.k.a.~the lattice of alternating sign matrices).
\end{abstract}

\newpage

\section{Introduction}

Our work is motivated by making the following theorems explicit for type A, where the finite Weyl group is the symmetric group, $\mathfrak{S}_n$ and the affine Weyl group is $\tilde{\mathfrak{S}}_n$.
\begin{theorem}{J.L.Waldspurger, 2005} \cite{waldspurger2007remarque}

 Let $W$ be a Weyl group presented as a reflection group on a Euclidean vector space $V$.  Let $C_{\omega}\subset V$ be the open cone over the fundamental weights and $C_R \subset V$ the closed cone spanned by the positive roots.  Let the cone associated with group element $g$ be $C_w:=(I-w)C_{\omega}$ (where $I$ is the identity element in $G$). One has the decomposition
 $$C_R=\displaystyle \bigsqcup_{w\in W} C_w$$ 
\end{theorem}

\begin{figure}
\centering
\includegraphics[scale=.7]{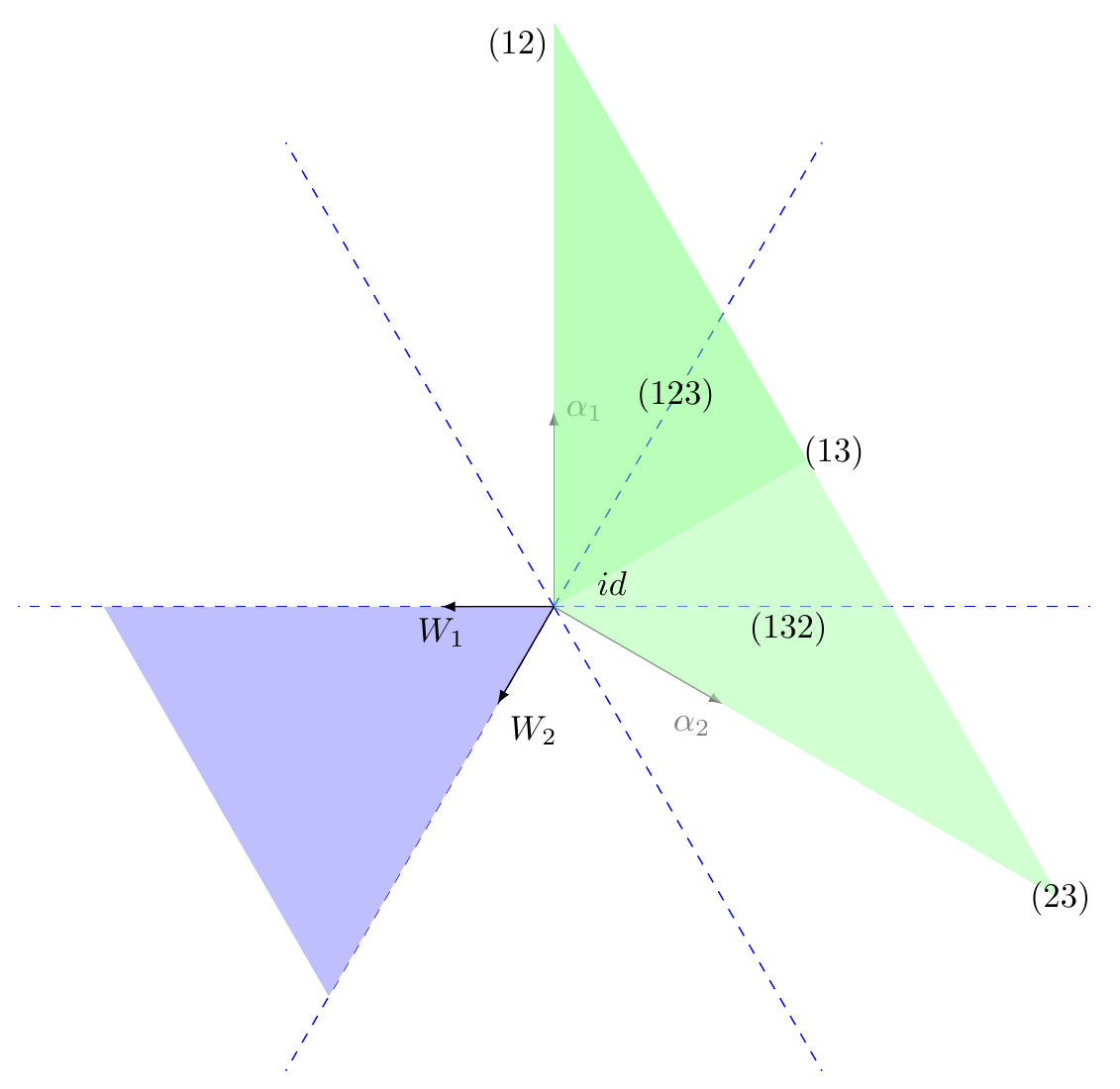}
\caption{The Waldspurger Decomposition for $A_2=\mathfrak{S}_3$}
\label{fig:a2wald}
\end{figure}

\begin{theorem}{E. Meinrenken, 2006} \cite{meinrenken2009tilings}\cite{bibikov2009tilings}

Let the affine Weyl group for a crystallographic Coxeter system be denoted $W^a$ and recall that
$W^a= \Lambda\rtimes W$ where the coroot lattice $\Lambda \subset V$ acts by translations. Let $A \subset C$ denote the Weyl
alcove, with $0 \in A$.  Then the images $V_w = (id-w)A$, $w \in W^a$ are all disjoint, and their union is all of $V$.
That is, $$V=\bigsqcup_{w\in W^a} V_w$$
\end{theorem}

\begin{figure}
\centering
\includegraphics[scale=.7]{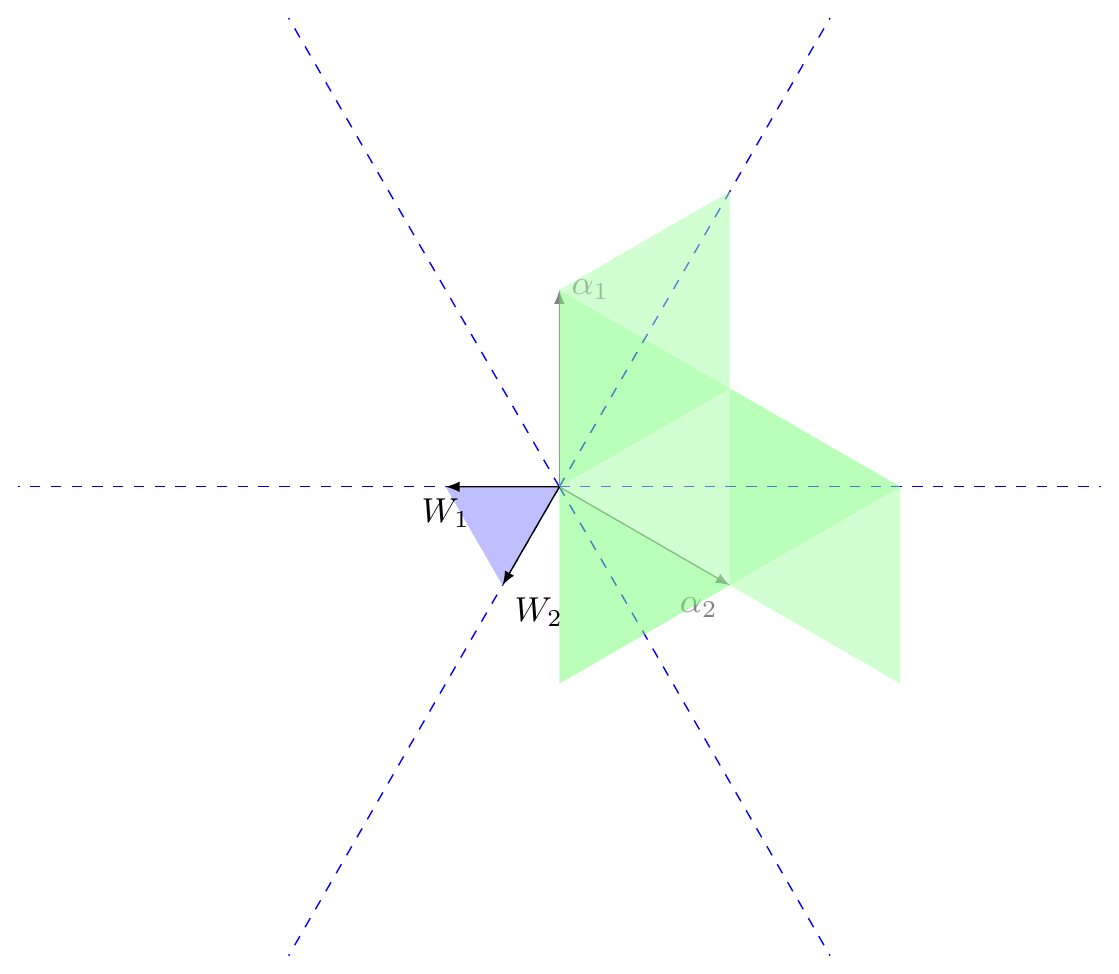}
\caption{The Meinrenken tiling for $A_2=\mathfrak{S}_3$.}
\label{fig:a2mein}
\end{figure}
We will define the \textbf{Meinrenken tile} to be $\displaystyle\bigsqcup_{w\in W} V_w$, restricting to a copy of the finite Weyl group inside of the affine Weyl group.  The semi-direct product of this finite Weyl group with the coroot lattice simply translates the Meinrenken tile and so this restriction is convenient from a combinatorial perspective.  Although it is built out of simplices, the Meinrenken tile is not a simplicial complex, nor even a CW complex, and it need not even be convex.

In type A, where our Weyl group elements are permutations, both $C_{\pi}=C_w$ and $V_{\pi}=V_w$ can be found in a purely combinatorial way by considering what we call the \textit{Waldspurger transform of the permutation} $\pi$, $\mathbf{WT}(\pi)$.  $\mathbf{WT}(\pi)$ is an $(n-1) \times (n-1)$ matrix constructed from the $n \times n$ permutation matrix via a \emph{transformation diagram} like the one at the top of Section \ref{fund trans section}  Section \ref{Section 2} is dedicated to defining the Waldspurger transform of a permutation, giving a combinatorial description of the transform, and verifying that the combinatorial description agrees with the definition.  The proof of Theorem \ref{wald transform thm} is fundamental, but not very illuminative, and the impatient reader is invited to skip to Section \ref{fund trans section}.  The remainder of the paper is organized as follows:

In Section \ref{geom obs section}  we informally discuss the geometry of the Meinrenken tile, particularly its symmetries and irregularities.  In Section \ref{um vectors section} we return to the combinatorics of Waldspurger matrices. We classify their row and column vectors by showing that they satisfy certain unimodality conductions and call them ``UM vectors''. We give explicit bijections between UM vectors and Unimodal Motzkin paths, Abelian ideals in the Lie algebra, $\mathfrak{sl}_n(\mathfrak{C})$, tableau with bounded hook lengths, and coroots in a certain polytope studied by Panyushev, Peterson, and Kostant \cite{panyushev2011abelian}.  In section \ref{entropy,asm, gen trans section} we show that componentwise comparison of Waldspurger matrices is Bruhat order and that summing all of the entries of the matrix gives the rank of the corresponding permutation in the lattice of alternating sign matrices (or monotone triangles).  Inspired by this, we extend the Waldspurger transform to alternating sign matrices and exhibit a lattice isomorphism between these generalized Waldspurger matrices and monotone triangles.  This lattice is known to be distributive, with join-irreducibles the bigrassmannian elements; permutations with exactly one left descent and one right descent \cite{sch}.  We show that these correspond to the Waldspurger matrices which are determined by a single entry.  In Section \ref{types b&c section} we explore types B and C.  We define the Waldspurger transform with respect to any crystalographic root system $\Phi$ and exhibit a combinatorial means of computing type B and C Waldspurger matrices by folding centrally symmetric type A Waldspurger matrices.  Here row and column vectors are not in bijection with abelian ideals and componentwise comparison is no longer Bruhat order. It is known that Dedekind-MacNeille completion of Bruhat order is still distributive for types B and C, and that the join-irreducibles are a strict subset of the bigrasmannian elements.  Various descriptions of the join irreducibles have been given in \cite{sch}\cite{geck1997bases} \cite{Reading2002}\cite{2016arXiv161208670A}.  Nevertheless, we present a conjectural description for these join-irreducibles: they correspond to the type C Waldspurger matrices specified by a single entry.
\section{The Waldspurger Transform for Permutations}\label{Section 2}

\begin{definition} Let $\phi$ denote the reflection representation of the symmetric group $$\phi:\mathfrak{S}_n \longrightarrow GL_{n-1}(\mathbb{R})$$ The Waldspurger matrix, $\mathbf{WT}(g)$, of a permutation $g$ is the matrix of $\phi(1)-\phi(g)$ applied to the matrix with columns given by the fundamental weights, expressed in root coordinates. 
\end{definition}

Our first theorem gives a concrete combinatorial method of finding the Waldspurger matrix associated with a given permutation.  It is helpful to consider the Cartan Matrix of the type A root system.
The \textit{Cartan matrix} of a root system is the matrix whose elements are the scalar products
$$a_{{ij}}=2 \frac {(r_{i},r_{j}) }{(r_{i},r_{i})}$$
(sometimes called the Cartan integers) where the $r_i$'s are the simple roots.  Recall that the root system $A_{n-1}$ has as its simple roots the vectors $a_i=e_i-e_{i+1}$ for $i=1,\dots , n-1$.  One can verify that the Cartan matrix for this root system has $2$'s on the main diagonal, $(-1)$'s on the superdiagonal and subdiagonal, and $0$s elsewhere.  Its columns express the simple roots in the basis of fundamental weights.
\begin{theorem} \label{wald transform thm}
Let $P$ be the $(n-1) \times (n-1)$ matrix for the permutation $\pi \in\ \mathfrak{S}_n$ expressed in root coordinates. Let $C$ be the $(n-1) \times (n-1)$ Cartan matrix and let $D$ be the $(n-1) \times (n-1)$ matrix

$$D_{i,j}= \begin{cases} 
      \displaystyle\sum_{\substack{a\leq i\\b>j}}\pi_{a,b} & i\leq j \\
      \displaystyle\sum_{\substack{a> i\\b\leq j}}\pi_{a,b} & i\geq j\\ 
   \end{cases}.
$$ 
Then we have that $\mathbf{(I-P)=DC}$.
\end{theorem}

\begin{proof}
We use the fact that $C=A^TA$ where $A$ is the $n \times (n-1)$ matrix $$A=\begin{pmatrix*} 1 & 0 & 0 & \dots & 0\\ 
-1 & 1 & 0 & \dots & 0\\
0 & -1 & 1 & \dots & 0\\
\vdots & \vdots & \vdots & \ddots & \vdots\\
0 & 0 & 0 & \dots & 1\\
0 & 0 & 0 & \dots & -1\end{pmatrix*}$$

to rewrite the conclusion :
$$P = I-DA^TA$$
We multiply both sides on the left by A:
$$AP = A-ADA^TA$$
We then make the observation that $AP=\pi A$.  Making this substitution and canceling the $A$'s on the right we obtain:
$$ \pi =I-ADA^T$$
This we will verify.
\newline

Multiplying $A$ and $D$, we see that $(AD)_{i,j}=D_{i,j}-D_{i-1,j}$ with the understanding $D_{0,k}:=0$ for all $k$.
One more multiplication gives us that $$(ADA^T)_{i,j}=D_{i,j}-D_{i-1,j}-D_{i,j-1}+D_{i-1,j-1}$$
once again, with the understanding that if either $i=0$ or $j=0$ then $D_{i,j}:=0$

\begin{case}
If $i=j$ then
\begin{align*} (ADA^T)_{i,j} &= D_{i,j}-D_{i-1,j}-D_{i,j-1}+D_{i-1,j-1} \\
&=  \sum_{\substack{a\leq i\\b>j}}\pi_{a,b} - \sum_{\substack{a\leq i-1\\b>j}}\pi_{a,b}- \sum_{\substack{a> i\\b\leq j-1}}\pi_{a,b}+ \sum_{\substack{a> i-1\\b\leq j-1}}\pi_{a,b}\\
&= \sum_{k \neq j} \pi_{i,k}\\
&= \begin{cases} 0 & \pi_{i,j}=1\\ 1 & \pi_{i,j}=0 \end{cases}
\end{align*}
\textrm{To understand the second-to-last inequality, observe that we are summing over the following terms of permutation matrices:}
\bigbreak
\resizebox{.25\textwidth}{!}{$\begin{pmatrix} 
\ddots & \vdots & \multicolumn{1}{c|}{\vdots} & \vdots & \udots  \\ 
\dots & \pi_{i-1,j-1} &\multicolumn{1}{c|}{ \pi_{i-1,j}} & \pi_{i-1,j+1} & \dots \\ 
\dots & \pi_{i,j-1} & \multicolumn{1}{c|}{\pi_{i,j}} & \pi_{i,j+1} & \dots \\ \cmidrule{4-5} 
\dots & \pi_{i+1,j-1} & \pi_{i+1,j} & \pi_{i+1,j+1} & \dots \\
\udots & \vdots     & \vdots    & \vdots      & \ddots \end{pmatrix}$}-
\resizebox{.25\textwidth}{!}{$\begin{pmatrix} 
\ddots & \vdots & \multicolumn{1}{c|}{\vdots} & \vdots & \udots  \\ 
\dots & \pi_{i-1,j-1} &\multicolumn{1}{c|}{ \pi_{i-1,j}} & \pi_{i-1,j+1} & \dots \\ \cmidrule{4-5}
\dots & \pi_{i,j-1} & \pi_{i,j} & \pi_{i,j+1} & \dots \\  
\dots & \pi_{i+1,j-1} & \pi_{i+1,j} & \pi_{i+1,j+1} & \dots \\
\udots & \vdots     & \vdots    & \vdots      & \ddots \end{pmatrix}$}-
\resizebox{.25\textwidth}{!}{$\begin{pmatrix} 
\ddots & \vdots & \vdots & \vdots & \udots  \\ 
\dots & \pi_{i-1,j-1} & \pi_{i-1,j} & \pi_{i-1,j+1} & \dots \\ 
\dots & \pi_{i,j-1} & \pi_{i,j} & \pi_{i,j+1} & \dots \\ \cmidrule{1-2} 
\dots & \pi_{i+1,j-1} & \multicolumn{1}{|c}{\pi_{i+1,j}} & \pi_{i+1,j+1} & \dots \\
\udots & \vdots     & \multicolumn{1}{|c}{\vdots}    & \vdots      & \ddots \end{pmatrix}$}+
\resizebox{.25\textwidth}{!}{$\begin{pmatrix} 
\ddots & \vdots & \vdots & \vdots & \udots  \\ 
\dots & \pi_{i-1,j-1} & \pi_{i-1,j} & \pi_{i-1,j+1} & \dots \\ \cmidrule{1-2}
\dots & \pi_{i,j-1} & \multicolumn{1}{|c}{\pi_{i,j}} & \pi_{i,j+1} & \dots \\  
\dots & \pi_{i+1,j-1} & \multicolumn{1}{|c}{\pi_{i+1,j}} & \pi_{i+1,j+1} & \dots \\
\udots & \vdots     & \multicolumn{1}{|c}{\vdots}    & \vdots      & \ddots \end{pmatrix}$}
\bigskip
Thus, $(I-ADA^T)_{i,j}=\pi_{i,j}$ for this case.
\end{case}
\bigskip

\begin{case}
If $i<j$ then
\begin{align*} (ADA^T)_{i,j} &= D_{i,j}-D_{i-1,j}-D_{i,j-1}+D_{i-1,j-1} \\
&=  \sum_{\substack{a\leq i\\b>j}}\pi_{a,b} - \sum_{\substack{a\leq i-1\\b>j}}\pi_{a,b}- \sum_{\substack{a\leq i\\b> j-1}}\pi_{a,b}+ \sum_{\substack{a\leq i-1\\b> j-1}}\pi_{a,b}\\
&= -\pi_{i,j}\\
\end{align*}
\textrm{This last equality is, again, easier to understand visually:}
\bigbreak
\resizebox{.25\textwidth}{!}{$\begin{pmatrix} 
\ddots & \vdots & \multicolumn{1}{c|}{\vdots} & \vdots & \udots  \\ 
\dots & \pi_{i-1,j-1} &\multicolumn{1}{c|}{ \pi_{i-1,j}} & \pi_{i-1,j+1} & \dots \\ 
\dots & \pi_{i,j-1} & \multicolumn{1}{c|}{\pi_{i,j}} & \pi_{i,j+1} & \dots \\ \cmidrule{4-5} 
\dots & \pi_{i+1,j-1} & \pi_{i+1,j} & \pi_{i+1,j+1} & \dots \\
\udots & \vdots     & \vdots    & \vdots      & \ddots \end{pmatrix}$}-
\resizebox{.25\textwidth}{!}{$\begin{pmatrix} 
\ddots & \vdots & \multicolumn{1}{c|}{\vdots} & \vdots & \udots  \\ 
\dots & \pi_{i-1,j-1} &\multicolumn{1}{c|}{ \pi_{i-1,j}} & \pi_{i-1,j+1} & \dots \\ \cmidrule{4-5}
\dots & \pi_{i,j-1} & \pi_{i,j} & \pi_{i,j+1} & \dots \\  
\dots & \pi_{i+1,j-1} & \pi_{i+1,j} & \pi_{i+1,j+1} & \dots \\
\udots & \vdots     & \vdots    & \vdots      & \ddots \end{pmatrix}$}-
\resizebox{.25\textwidth}{!}{$\begin{pmatrix} 
\ddots & \vdots & \multicolumn{1}{|c}{\vdots} & \vdots & \udots  \\ 
\dots & \pi_{i-1,j-1} &\multicolumn{1}{|c}{ \pi_{i-1,j}} & \pi_{i-1,j+1} & \dots \\ 
\dots & \pi_{i,j-1} & \multicolumn{1}{|c}{\pi_{i,j}} & \pi_{i,j+1} & \dots \\ \cmidrule{3-5} 
\dots & \pi_{i+1,j-1} & \pi_{i+1,j} & \pi_{i+1,j+1} & \dots \\
\udots & \vdots     & \vdots    & \vdots      & \ddots \end{pmatrix}$}+
\resizebox{.25\textwidth}{!}{$\begin{pmatrix} 
\ddots & \vdots & \multicolumn{1}{|c}{\vdots} & \vdots & \udots  \\ 
\dots & \pi_{i-1,j-1} &\multicolumn{1}{|c}{ \pi_{i-1,j}} & \pi_{i-1,j+1} & \dots \\ \cmidrule{3-5}
\dots & \pi_{i,j-1} & \pi_{i,j} & \pi_{i,j+1} & \dots \\  
\dots & \pi_{i+1,j-1} & \pi_{i+1,j} & \pi_{i+1,j+1} & \dots \\
\udots & \vdots     & \vdots    & \vdots      & \ddots \end{pmatrix}$}
\bigskip
Thus, $(I-ADA^T)_{i,j}=\pi_{i,j}$ for this case as well.
\end{case}
\bigskip

\begin{case}
If $i>j$ then
\begin{align*} (ADA^T)_{i,j} &= D_{i,j}-D_{i-1,j}-D_{i,j-1}+D_{i-1,j-1} \\
&=  \sum_{\substack{a> i\\b\leq j}}\pi_{a,b} - \sum_{\substack{a> i-1\\b\leq j}}\pi_{a,b}- \sum_{\substack{a> i\\b\leq j-1}}\pi_{a,b}+ \sum_{\substack{a> i-1\\b\leq j-1}}\pi_{a,b}\\
&= -\pi_{i,j}\\
\end{align*}
\textrm{As before, the final equality is apparent with a visual:}
\bigbreak
\resizebox{.25\textwidth}{!}{$\begin{pmatrix} 
\ddots & \vdots & \vdots & \vdots & \udots  \\ 
\dots & \pi_{i-1,j-1} & \pi_{i-1,j} & \pi_{i-1,j+1} & \dots \\ 
\dots & \pi_{i,j-1} & \pi_{i,j} & \pi_{i,j+1} & \dots \\ \cmidrule{1-3} 
\dots & \pi_{i+1,j-1} & \multicolumn{1}{c|}{\pi_{i+1,j}} & \pi_{i+1,j+1} & \dots \\
\udots & \vdots     & \multicolumn{1}{c|}{\vdots}    & \vdots      & \ddots \end{pmatrix}$}-
\resizebox{.25\textwidth}{!}{$ \begin{pmatrix} 
\ddots & \vdots & \vdots & \vdots & \udots  \\ 
\dots & \pi_{i-1,j-1} & \pi_{i-1,j} & \pi_{i-1,j+1} & \dots \\ \cmidrule{1-3}
\dots & \pi_{i,j-1} & \multicolumn{1}{c|}{\pi_{i,j}} & \pi_{i,j+1} & \dots \\  
\dots & \pi_{i+1,j-1} & \multicolumn{1}{c|}{\pi_{i+1,j}} & \pi_{i+1,j+1} & \dots \\
\udots & \vdots     & \multicolumn{1}{c|}{\vdots}    & \vdots      & \ddots \end{pmatrix}$
}-
\resizebox{.25\textwidth}{!}{$\begin{pmatrix} 
\ddots & \vdots & \vdots & \vdots & \udots  \\ 
\dots & \pi_{i-1,j-1} & \pi_{i-1,j} & \pi_{i-1,j+1} & \dots \\ 
\dots & \pi_{i,j-1} & \pi_{i,j} & \pi_{i,j+1} & \dots \\ \cmidrule{1-2} 
\dots & \pi_{i+1,j-1} & \multicolumn{1}{|c}{\pi_{i+1,j}} & \pi_{i+1,j+1} & \dots \\
\udots & \vdots     & \multicolumn{1}{|c}{\vdots}    & \vdots      & \ddots \end{pmatrix}$}+
\resizebox{.25\textwidth}{!}{$\begin{pmatrix} 
\ddots & \vdots & \vdots & \vdots & \udots  \\ 
\dots & \pi_{i-1,j-1} & \pi_{i-1,j} & \pi_{i-1,j+1} & \dots \\ \cmidrule{1-2}
\dots & \pi_{i,j-1} & \multicolumn{1}{|c}{\pi_{i,j}} & \pi_{i,j+1} & \dots \\  
\dots & \pi_{i+1,j-1} & \multicolumn{1}{|c}{\pi_{i+1,j}} & \pi_{i+1,j+1} & \dots \\
\udots & \vdots     & \multicolumn{1}{|c}{\vdots}    & \vdots      & \ddots \end{pmatrix}$}
\bigskip
Thus, $(I-ADA^T)_{i,j}=\pi_{i,j}$ in this final case.
\end{case}
\end{proof}
Because the inverse of the Cartan matrix expresses the fundamental weights in simple root coordinates, we may multiply both sides of the equation above on the right by $\mathbf{C}^{-1}$ and observe
$$\mathbf{D}=\mathbf{WT}(\pi).$$
\subsection{The Fundamental Transformation}\label{fund trans section}

Let $\pi \in \mathfrak{S}_n$ be expressed as an $n \times n$ permutation matrix. For aesthetics, our examples put the entries of $\pi$ on a grid, leave off the zeros, and use stars instead of ones.  Construct the $(n-1) \times (n-1)$ Waldspurger matrix $\mathbf{WT}(\pi)$ in the spaces between the entries of the permutation matrix as follows:
If an entry is on or above the main diagonal, count the number of stars above and to the right, and put that count in the space.  If the entry is on or below the main diagonal, count the number of stars below and to the left and put that count in the space.  Note that entries on the diagonal are still well-defined.
As an example, here is the Waldspurger matrix for the permutation $456213\in \mathfrak{S}_6$.

\begin{center}
\includegraphics[]{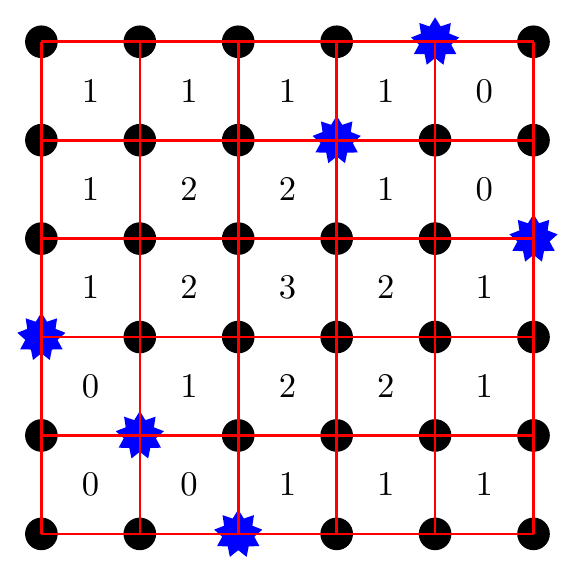}
\end{center}

Now suppose $M$ is a Waldspurger matrix for the permutation $\pi$, with columns $c_1,c_2,\dots,c_{n-1}$. To return to the language of the Waldspurger and Meinrenken theorems we have: 
\begin{equation}
C_M:=C_{\pi}=\left\{ \displaystyle\sum_{i=1}^{n-1} a_ic_i \vert \hspace{10pt} a_i\in \mathbb{R}_{\geq 0} \right\}
\end{equation}
\begin{equation}V_M:=V_{\pi}=\left\{\displaystyle\sum_{i=1}^{n-1} a_ic_i \vert \hspace{10pt} a_i \in \mathbb{R}_{\geq 0}\textrm{ and }\sum a_i\leq1 \right\}
\end{equation}

It is at times convenient to study the boundary of the Meinrenken tile, so we will also define
\begin{equation}\Delta_M:=\Delta_{\pi}:=\left\{\displaystyle\sum_{i=1}^{n-1} a_ic_i \vert \hspace{10pt}a_i \in \mathbb{R}_{\geq 0}\textrm{ and }\sum a_i=1\right\}
\end{equation}

\section{Geometric Observations}\label{geom obs section}
 Our first example, in Figures \ref{fig:a2wald} and \ref{fig:a2mein}, was in many ways too nice.  One may be tempted to study the Meinrenken tile or a slice of the root cone as a simplicial complex, or at the very least a regular CW complex.  Going up even one dimension presents several unforeseen complications.  For starters, our Meinrenken tile is no longer convex! See, for example, the right side of Figure \ref{meinzome}, constructed out of zometools. (The two yellow edges and one blue edge coming out from the origin are the fundamental weights.)
 \begin{figure}
 \begin{center}
 \begin{minipage}[t][][b]{.3\textwidth}
 \includegraphics[scale=1.5]{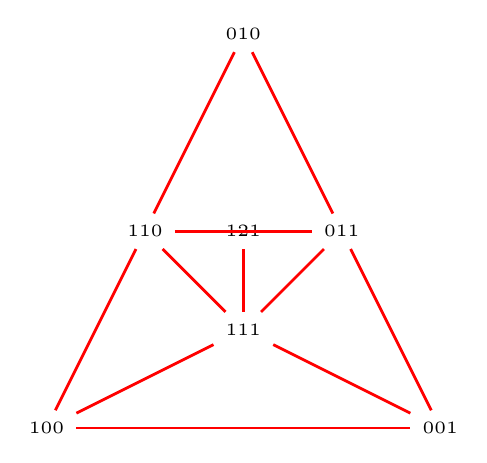}
 \end{minipage}
 \hspace{100pt}
 \begin{minipage}[t][][b]{.3\textwidth}
 \includegraphics[scale=.070]{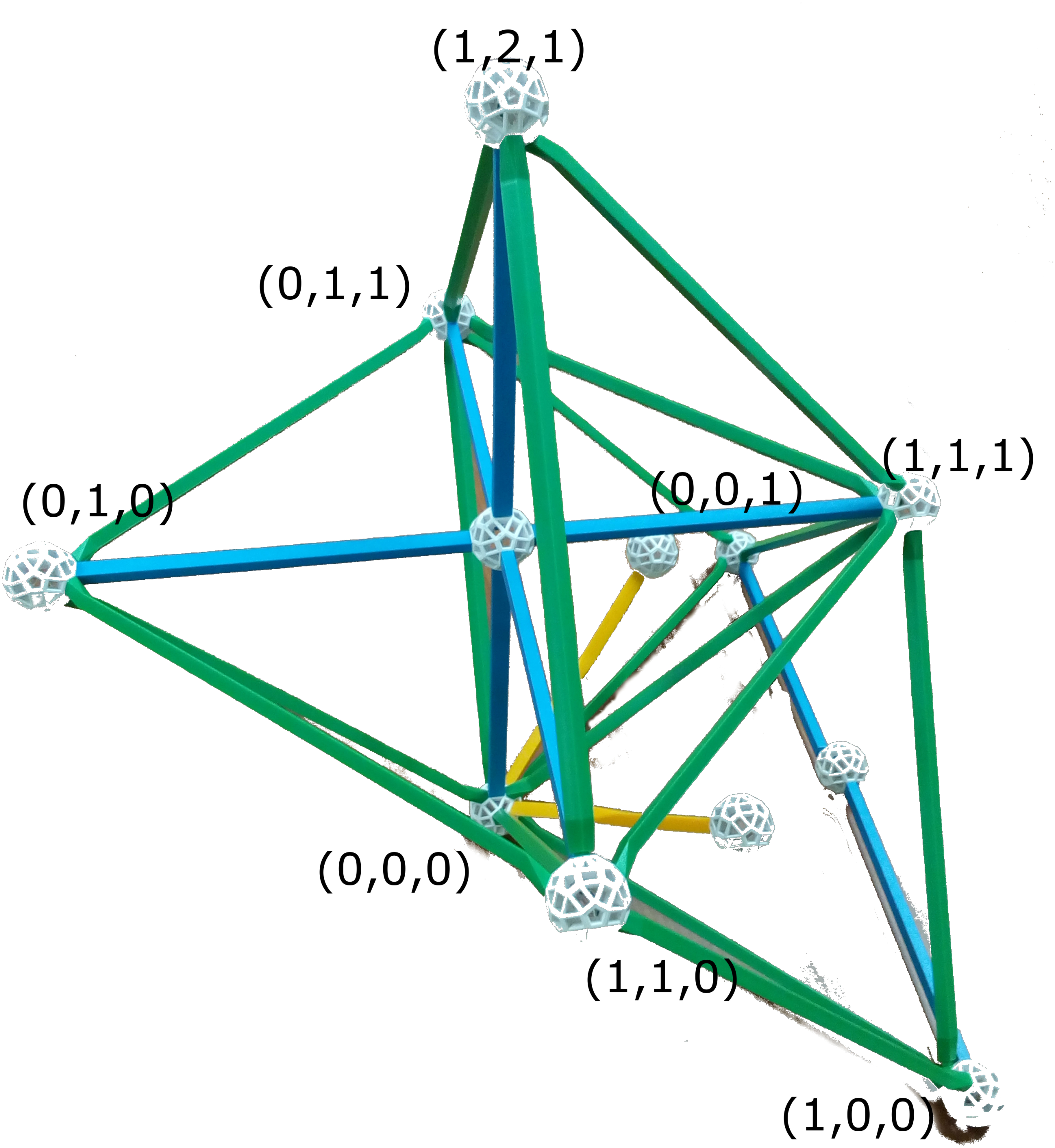}
 \end{minipage}
\end{center}
\label{meinzome}
\caption{A slice of the root cone of type $A_3$ with points labeled in root coordinates, along with the corresponding $A_3$ Meinrenken tile.}
  \end{figure}
 Observe from the left side of figure that the slice of the root cone fails to be simplicial or regular CW complex.  The top triangle intersects the two below it along ``half edges".  One may desire to consider it instead as a degenerate square to fix this impediment, but from the Meinrenken tile, it seems this new vertex should rightly be the fundamental weight with root coordinates $(\frac{1}{2},1,\frac{1}{2})$ and not the vertex $(1,2,1)$. If we wish to proceed in this manner, we must then include $(\frac{1}{2},1,\frac{1}{2})$ as a vertex for the two triangles 110,121,111 and 111,121,011 and consider them as degenerate tetrahedra.  This sort of topological completion via intersecting facets has proven to be a rabbit hole with less fruit than one might hope for.  Instead let us turn our attention back to the symmetric group, and consider Figure \ref{oneline}.
 
 \begin{figure}
 \begin{center} 
  \includegraphics[scale=.90]{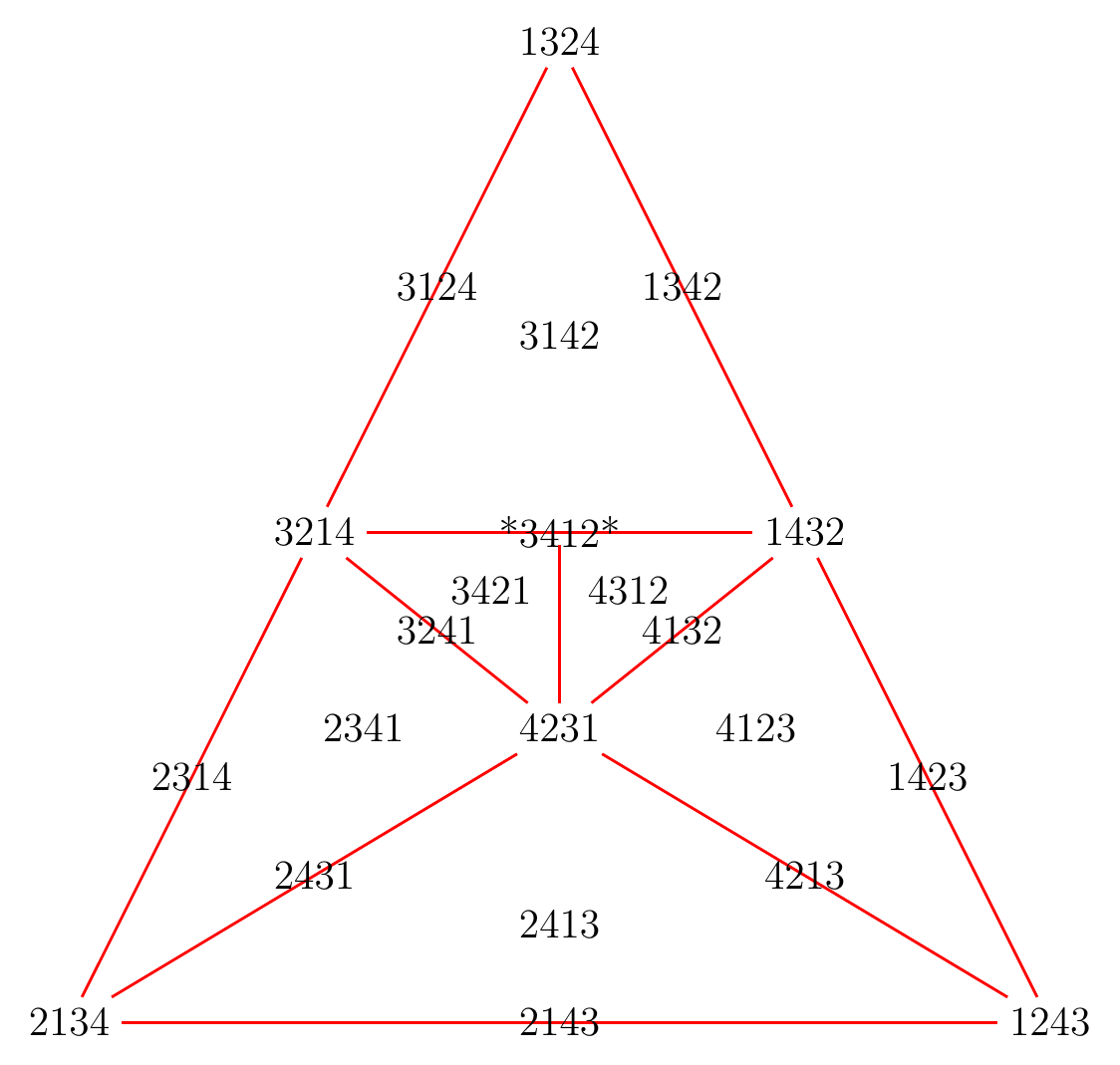}
\includegraphics[scale=.90]{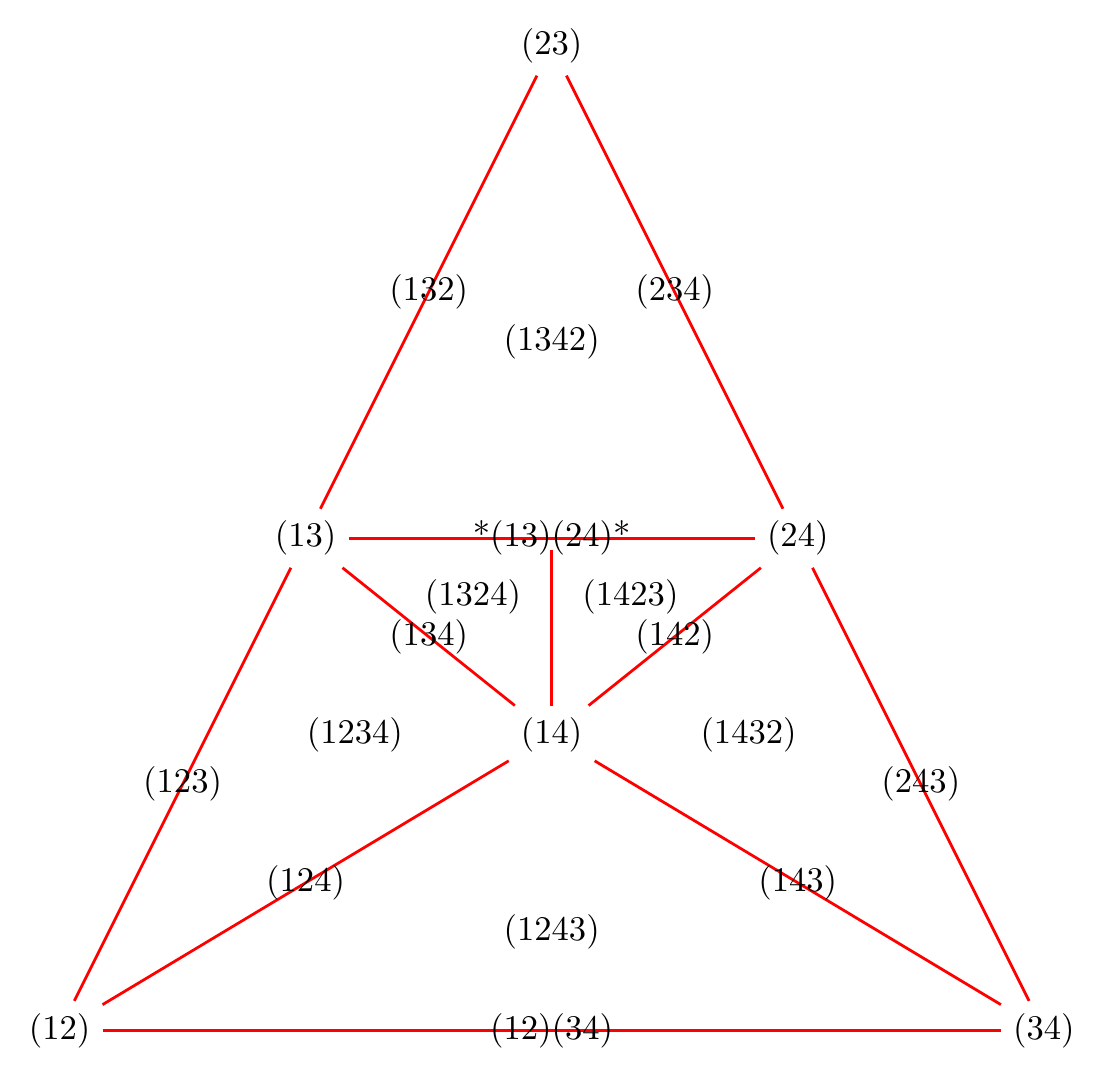}
 \caption{A slice of the Root cone $A_3=\mathfrak{S}_4$}
 \label{oneline}

 \end{center}
 \end{figure}
  Observe that the dimension of a simplex in the cone slice relates to the number of cycles (counting fixed points as one-cycles) of the corresponding permutation.  The four-cycles are the triangles, the three-cycles and disjoint two-cycles are the edges, and the transpositions are vertices.  This has been known for some time \cite{bibikov2009tilings} and can be seen as a corollary to the Chevalley-Shephard-Todd theorem \cite{Bourbaki:2008:LGL:1502204}.
 The astute observer will notice that there are two permutations missing in the picture.  The identity corresponds to the cone point which we cut off, and the vertical edge in the center we left unlabeled, as we feel that it (along with the starred edge $3412$) deserves some discussion.  It corresponds to the permutation $4321$ and its $3 \times 3$ Waldspurger matrix has all entries equal to one except for a two in the middle.  If we consider the columns of each Waldspurger matrix as being ordered from left to right, the cones in the Waldspurger decomposition are endowed with an orientation.  The orientation appears to be consistent, but what does it say in the case of this permutation?  It appears to go first up from $(1,1,1)$ to $(1,2,1)$ and then back down.  The starred edge, $3412$ is also worth mentioning.  Its Waldspurger matrix has first column $(1,1,0)$ second column $(1,2,1)$ and third column $(0,1,1)$ so it is perhaps better seen as a degenerate triangle than as an edge.  Looking at the Meinrenken tile, we see that $\Delta_{3,4,1,2}$ is actually a triangle.  The strangeness in the Meinrenken picture comes from the fact that $V_{3,4,1,2}$ is a square and not a tetrahedron.
 
 Despite all of these collapses in dimension, there is still a fair amount of symmetry in the Meinrenken tile.
 \begin{theorem}
Let $R$ denote reflection through the affine hyperplane orthogonal the longest positive root, $\theta$, at height one.  Then $R$ is an involution on the set of $\Delta_{\pi}$'s. At the level of permutations, this involution is just applying the transposition $(1,n)$ on the left. 
$$R( \Delta_{\pi})=\Delta_{(1,n)\pi}.$$

 In contrast, applying the transposition $(1,n)$ on the right is the gluing map for using multiple Meinrenken tiles to tile space.  The left right symmetry is conjugation by the longest element in the Coxeter group.

\end{theorem}
\begin{proof}
Conjugation by the longest element is known to induce a left right symmetry in the root lattice.  In the next section we will see that UM vectors (the columns of Waldspurger matrices) are really special order filters in this lattice, and inherit this same action.\newline 
To prove the other two statements, we will consider how the transformation diagram changes when one applies the transposition $(1,n)$ on the left (respectively right).  The two moving stars will cause $\theta$ to be subtracted from all columns (respectively rows) starting and ending with $1$'s, and to be added to all columns (respectively rows) starting and ending with $0$'s. Adding or subtracting $\theta$'s from rows is acting by translation on the $\Delta_{\pi}$'s, and since this transformation preserves being a Waldspurger matrix, it must be the gluing map for attaching multiple Meinrenken tiles to tile space.\newline
In contrast, adding $\theta$'s to columns is the reflection $R$.
Indeed, consider where $R$ sends column vectors.  If we let $P$ denote projection onto $\theta$, then $v \mapsto (id-2P)v+\theta$.  In root coordinates, this projection is described by the matrix $\frac{2\theta \theta^{T}C}{\theta^{T}C\theta}=JC$ where $J$ is the matrix of all ones and $C$ is the Cartan matrix.  One may verify that $$JC=\left[\begin{matrix}1&0&\dots&0&1\\1&0&\dots&0&1\\ \vdots&\vdots&\vdots&\vdots &\vdots\\1&0&\dots&0&1\\1&0&\dots&0&1 \end{matrix}\right]$$
and thus $$v \mapsto (I-JC)v+\theta =v-(v_1+v_{n-1})\theta +\theta= \begin{cases}
v & \textrm{ if } v_1+v_{n-1}=1\\
v+\theta & \textrm{ if } v_1+v_{n-1}=0 \\
v-\theta & \textrm{ if } v_1+v_{n-1}=-1.
\end{cases}
$$

\end{proof}

\section{UM vectors}\label{um vectors section}
Suppose $v$ is the $k$-th column of the Waldspurger matrix associated to the permutation $\pi$.  It is evident from the transformation diagram that $v_1=0$ or $v_1=1$ since the one in the first row of $\pi$ can either be to the left or to the right of $v_1$.  By similar reasoning, for $i\leq k$ we have $v_i=v_{i-1}$ or $v_i=v_{i-1}+1$  and for $i>k$ we have $v_i=v_{i-1}$ or $v_i=v_{i-1}-1$ with $v_n=0$ or $v_n=1$. In other words, $v$ will start with a zero or a one, weakly increase (by steps of $0$ or $1$) until the $k$th entry, and then weakly decrease (by steps of $0$ or $1$), to the last entry.

\begin{definition}
A \emph{Motzkin path} is a lattice path in the integer plane $\mathbb{Z}\times \mathbb{Z}$ consisting of steps $(1,1), (1,-1),(1,0)$ which starts and ends on the x-axis, but never passes below it. A Motzkin path is \emph{unimodal} if all occurrences of the step $(1,1)$ are before the occurrences of $(1,-1)$.  For brevity, we will hence forth refer to unimodal Motzkin paths as \emph{UMP}'s.
\end{definition}
\begin{lemma} 
\emph{(counting UMPs)} \newline
There are $2^{n-1}$ UMPs between $(0,0)$ and $(0,n)$.
\end{lemma}
\begin{proof}\emph{(induction)}
\newline
{\bf Base case: } There is only one UMP of length one, and only two UMPs of length two.
\newline

{\bf Induction hypotheses:}
Suppose there are $2^{k-1}$ UMPs of length $k$ for all $k\leq n-1$.  Consider an arbitrary UMP of length $n$.\\

{\bf Case 1:}\\ 
The first step of the UMP is $(1,0)$. Cutting off this step, we have an arbitrary UMP of length $n-1$ and so by induction, that there are $2^{n-2}$ such UMPs.\\

\begin{center}
\includegraphics[]{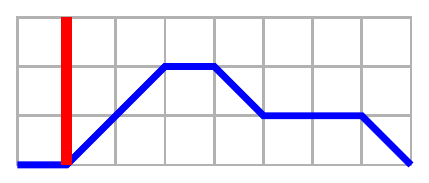}
\end{center}

{\bf Case 2:}\\ 
The last step of the UMP is $(1,0)$. Cutting off this last step, we have an arbitrary UMP of length $(n-1)$ and so by induction, that there are $2^{n-2}$ such UMPs.\\

\begin{center}
\includegraphics[scale=.5]{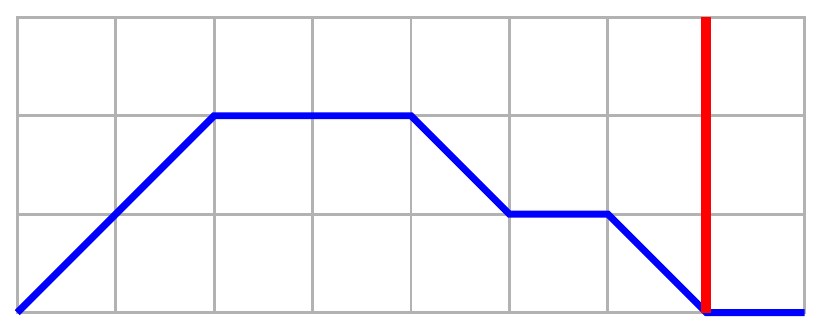}
\end{center}

{\bf Case 3:}\\ 
We have double counted some stuff.  If the first and last steps of a UMP are both $(1,0)$ then the UMP was counted by both of the previous cases.  Cutting off both the first and last steps we have an arbitrary UMP of length $n-2$.  There are, by induction, $2^{n-3}$ such UMPs.\\

\begin{center}
\includegraphics[scale=.5]{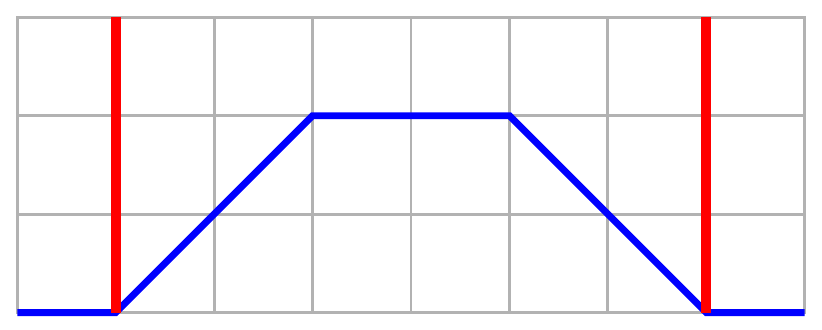}
\end{center}

{\bf Case 4:}\\ 
The first and last steps of the UMP are $(1,1)$ and $(1,-1)$, respectively. Cutting these steps once again, we see by induction, that there are $2^{n-3}$ such UMPs.\\

\begin{center}
\includegraphics[scale=.5]{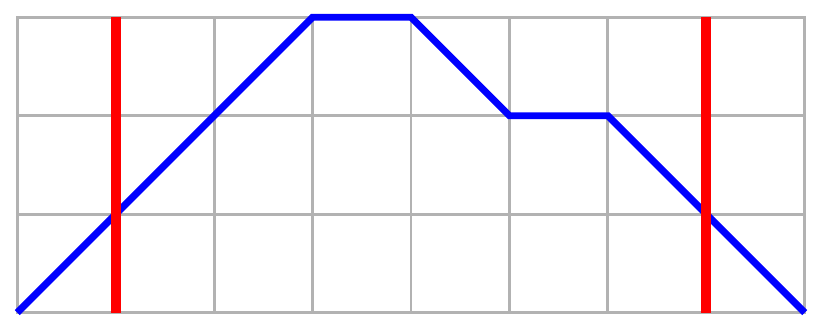}
\end{center}

So we see that there are $2^{n-2}+2^{n-2}-2^{n-3}+2^{n-3}=2^{n-1}$ UMPs of length $n$.
\end{proof}
\begin{definition}
A \emph{UM vector} is any vector that appears as a column in $\mathbf{WT}(\pi)$ for some permutation $\pi$.
\end{definition}

\begin{theorem}
There is a bijective correspondence between UM vectors of length $n-1$ and UMPs with $n$ steps.  Consequentially, there are  $2^n$ UM vectors of length $n$.
\end{theorem}
\begin{proof}
A UM vector must start with a zero or a one, weakly increase by one until its entry on the diagonal, and then weakly decreases by one until its final entry, a zero or one.  Any row vector of a Waldspurger matrix must also be a UM vector with its maximum also on the diagonal. Padding a UM vector with zeros on each end gives the $x$ coordinates for a UMP of length $n$.
For example, 
\begin{center}
$(1,2,3,3,2,2,1)\leftrightarrow(0,1,2,3,3,2,2,1,0)\leftrightarrow$
\includegraphics[]{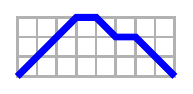}
\end{center}
\end{proof}

\begin{theorem}
UM vectors are in bijection with tableaux with hook length bounded above by $n$ and with Abelian ideals in the nilradical of the Lie Algebra $\mathfrak{sl}_n$.
\end{theorem}
\begin{proof}

One can take any UM vector and write it as a sum of positive roots by recursively subtracting the highest root whose nonzero entries correspond to positive nondecreasing entries in the UM vector.  For example, the vector $(0,1,2,1)=(0,1,1,0)+(0,0,1,1)$  This set of positive roots will always generate an abelian ideal in the nilradical of the Lie Algebra $\mathfrak{sl}_n$ and will correspond to a tableau with bounded hook length, as seen in the diagram below.
\begin{center}
\includegraphics[scale=.8]{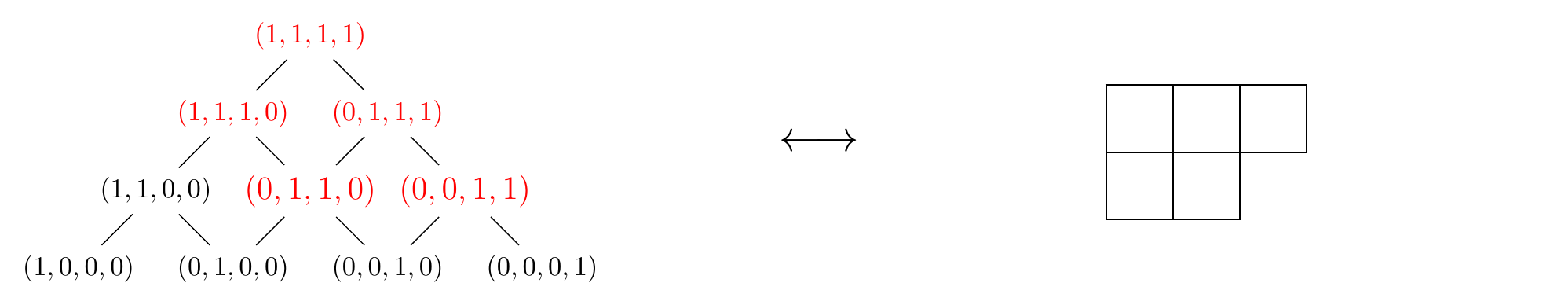}
\end{center}
\end{proof}
\begin{theorem}
 UM vectors are exactly the coroots $c$ (in root coordinates) such that $-1\leq (c,r)\leq 2$ for every positive root $r$.
 They are coroots inside the polytope defined by affine hyperplanes at heights $-1$ and $2$ orthogonal to every positive root. 
 \label{inequality}
\end{theorem}
\begin{proof}This follows from a result which Panyushev attributes to  Peterson and
Kostant \cite{panyushev2011abelian} which is expressed in the language of Abelian ideals. They show that the number of coroots inside of this ``Pederson Polytope" is $2^{n-1}$.  If we can show that our $2^{n-1}$ UM vectors are inside the polytope, we will be done.
Explicitly, suppose that $\bar{x}$ is a UM vector and $\bar{y}$ is a positive root (both expressed in root coordinates).  Then $(x,y)=x^t\cdot y=\bar{x}^tA^tA\bar{y}=\bar{x}^tC\bar{y}$ where $A$ is the matrix defined in Theorem \ref{wald transform thm} and $C$ is the Cartan matrix.  Suppose that $\bar{y}=(0,\dots,0,1,\dots,1,0,\dots 0)^t$ where the first one is in position $i$ and the last one is in position $j$.
Then 
\begin{align*}
\bar{x}^tC\bar{y}&=2\left(\sum_{k=i}^{j} x_k\right) -x_{i-1}-x_{j+1}-2\left(\sum_{k=i+1}^{j-1}x_k\right)\\
&=-x_{i-1}+x_{i}+x_{j}-x_{j+1}
\end{align*}
Because $(x_1,\dots,x_{n-1})$ is a UM vector, $x_i$ and $x_{i-1}$ can differ by at most one, and likewise $x_j$ and $x_{j+1}$ can differ by at most one.
This yields that $$-2\leq x_{i}-x_{i-1}+x_{j}-x_{j+1} \leq 2.$$  However, the $-2$ is unatainable by the unimodality of UM vectors. Suppose that $x_{i-1}>x_{i}$, that is $x_{i-1}=x_{i}+1$.  Then $x_j \geq x_{j+1}$, that is, $x_j=x_{j+1}$ or $x_j+1=x_{j+1}$.
Either way, $x_{i}-x_{i-1}+x_{j}-x_{j+1}=x_{i}-(x_{i}+1)+x_{j}-x_{j+1}>-2$.
Thus $$-1\leq x_{i}-x_{i-1}+x_{j}-x_{j+1} \leq 2$$ showing that our UM vectors are all inside the Pederson polytope.
\end{proof}

\section{Entropy, Alternating Sign Matrices, and the Waldspurger Transform in General}\label{entropy,asm, gen trans section}
\begin{definition}
The \textit{Entropy} (alternatively called variance in the literature) of a permutation $\pi$ is $$E(\pi):=\sum_{i=1}^n (\pi(i)-i)^2$$
\end{definition}
\begin{definition}
The \textit{Waldspurger Height} of a permutation $\pi$, is $$h(\pi):=\sum_{i=1}^n\sum_{j=1}^n \mathbf{WT}(\pi)_{i,j}$$
\end{definition}
\begin{theorem}\label{entropy} For $\pi \in \mathfrak{S}_n$,
$$h(\pi)=\frac{1}{2}E(\pi)$$.
\end{theorem}
\begin{proof}
Consider what each ``star" in the transformation diagram contributes to the Waldspurger matrix.  We can see it as contributing ones to every entry enclosed in the right triangle between itself and the main diagonal, and contributing one half for every entry on the main diagonal whose box is cut by the triangle. 
\begin{center}
\includegraphics[]{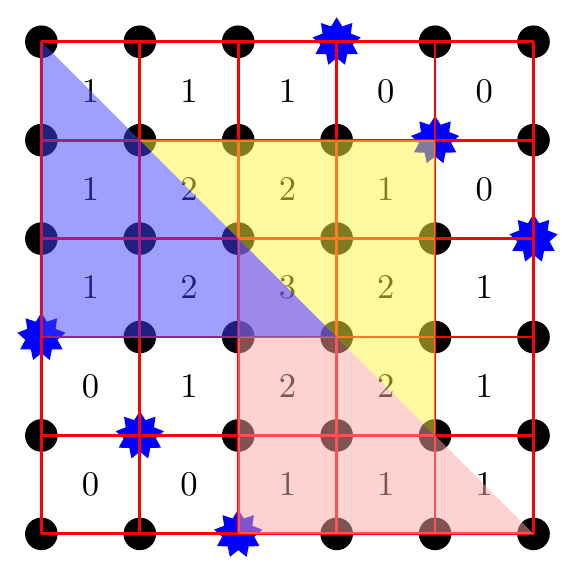}
\end{center}
\end{proof}
\begin{definition}
\emph{Alternating Sign Matrices} or ASMs, are square matrices with entries $0$, $1$, or $-1$ whose rows and columns sum to $1$ and alternate in sign.
\end{definition}
\begin{theorem}{A. Lascoux and M. Sch{\"u}tzenberger, 1996 \cite{sch}}\label{lascoux}

One half the entropy of a permutation is its rank in the Dedekind-MacNeille completion of the Bruhat order.  The elements in this lattice can be viewed as alternating sign matrices with partial order given by component-wise comparison of entries in their associated monotone triangles.  
\end{theorem}
The Dedekind-MacNeille completion of a poset $P$ is defined to be the smallest lattice containing $P$ as a subposet \cite{MacNeilleByBirkhoff}.  Its construction is similar to the Dedekind cuts used to construct the real numbers from the rationals.  For more on alternating sign matrices, monotone triangles, and their history we refer to ~\cite{asm}.  
This connection to alternating sign matrices motivates us to extend our definition of the Waldspurger transform to a larger class of matrices.

\begin{definition}
An $n \times n$ matrix $M$ is \textit{sum-symmetric} if its $i$th row sum equals its $i$th column sum for all $1\leq i \leq n$.  We write $M \in SS_n$.
\end{definition}

\begin{definition}\label{WT}
From an $n \times n$ sum-symmetric matrix $M$, define the $n-1 \times n-1$ matrix, $\mathbf{WT}(M)$ where $$\mathbf{WT}(M)_{i,j}=\begin{cases} 
      \displaystyle\sum_{\substack{a\leq i\\b>j}}M_{a,b} & i\leq j \\
      \displaystyle\sum_{\substack{a> i\\b\leq j}}M_{a,b} & i\geq j\\ 
   \end{cases}.
$$ 
\end{definition}

\begin{proposition}
$\mathbf{WT}(M)$ is well-defined if and only if $M \in SS_n$.  If $M$ were not sum-symmetric, the diagonal would be ``over-determined.''
\end{proposition}

\begin{proposition}The $\mathbf{WT}$ map is linear and surjective with kernel the diagonal matrices.
$$\mathbf{WT}:SS_n \twoheadrightarrow \textnormal{Mat}_{n-1}$$
\end{proposition}

\begin{theorem}
The restriction of the Waldspurger transform to alternating sign matrices has as its image all $M \in \textnormal{Mat}_{n-1}$ such that columns and rows of $M$ are UM vectors with maximums on the diagonal.  Component-wise comparison of these matrices is exactly the same order as is defined on the ASM lattice via monotone triangles.
\end{theorem}
\begin{proof}
See the next subsection.
\end{proof}


\subsection{The Lattice of Monotone Triangles}\label{MT lattice section}
We have a bijection between ASMs and generalized Waldspurger matrices and would like to show that componentwise comparison of generalized Waldspurger matrices is the same partial order as the componentwise comparison of monotone triangles (and is thus Dedekind-MacNeille completion of Bruhat order by Theorem \ref{lascoux}).  To this end, we consider the well-known bijection between monotone triangles and alternating sign matrices \cite{2014arXiv1408.5391S} obtained by letting the $k$th row of the triangle equal the positions of $1$'s in the sum of the first $k$ rows of an alternating sign matrix.  In particular, the identity matrix will always correspond to the monotone triangle
$$\begin{matrix}
1\\
1&2\\
\cdots\\
1&2&\cdots& n\\
\end{matrix}.$$
  Because this is the $\hat{0}$ in the lattice of monotone triangles and the partial order is componentwise comparison, we may consider \emph{reduced monotone triangles} by subtracting this triangle from all of the others (see figure \ref{four in the same}).

We will explicitly describe the composition of these two bijections and show that it is a poset isomorphism, preserving component-wise comparison. The map is easy to describe, but it will take a little work to verify that it is well-defined and surjective.
The map, from monotone triangles to Waldspurger matrices is as follows:  Subtract off the monotone triangle corresponding to the identity permutation, and then consider the entries of this reduced monotone triangle as ``painting instructions.''  The $(i,j)$th entry of the reduced triangle tells us how much paint to load our brush with for a left-to-right stroke beginning at the $(i,j)$th entry of the corresponding Waldspurger matrix.
As a working example, consider Figure \ref{four in the same}.  The two at the top of the reduced triangle is ``painted''
onto the $(1,1)$ and $(1,2)$ entries of the associated Waldspurger matrix.  The one in the next row is painted onto the $(2,1)$ entry, and the two after it is painted onto the $(2,2)$ and $(2,3)$ entries.

We must check that our painting gives a matrix with unimodal rows and columns with maximums on the diagonal.  The left-to-right painting process ensures that the entries in each row of the Waldspurger matrix will increase weakly by one up to the diagonal.  The fact that rows of the reduced triangle are weakly increasing guarantees that the row of the Waldspurger matrix will be weakly decreasing by ones after the diagonal.  The conditions on the columns are a bit more disguised, but the fact that reduced monotone triangles increase weakly up columns guarantees that the columns of the Waldspurger matrix will increase weakly up to the diagonal. Finally, the fact that reduced monotone triangles decrease by at most one in the $\searrow$ direction, guarantees that the columns of the Waldspurger matrix will decrease weakly above the diagonal.
This follows from induction on the size of the monotone triangle.  Suppose that the lower-left corner or the monotone trianges maps onto a generalized Waldspurger matrix of dimension one less.  Then painting a new diagonal will preserve the unimodality in rows and columns, and keep the maximums on the diagonal.

This painting map has an inverse ``peeling'' operation.  UM vectors by themselves are not in bijection with rows of reduced monotone triangles, but, if one knows that the UM vector is to appear in row $k$, our painting map will have an inverse ``peeling'' operation into $k$ entries as seen in Figure \ref{peel}.

To peel a UM vector into $k$ parts, create a diagram as in Figure \ref{peel} and specify $k$ starting points, one at the top of each of the $k$ columns.  First draw a path from the $k$th starting point to the end, staying as far up and to the right as possible.  Then do the same with the $(k-1)$st point.  Note that the unimodality condition on the UM vector guarantees that this path will be weakly shorter than the first one.  Continue in this way until all of the vertices are exhausted.  Record the length of the paths to get the corresponding row in the associated reduced monotone triangle.

\begin{figure}
\begin{center}
\includegraphics[scale=.45]{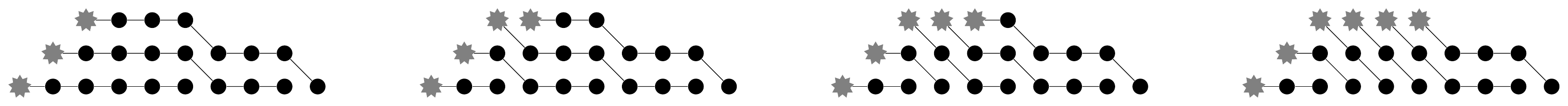}
\caption{There are four ways to peel the UM vector $1233332221$.  It may peel into three, four, five, or six parts, depending on which $3$ is on the diagonal of the Waldspurger matrix it is appearing in.}
\label{peel}
\end{center}
\end{figure}

\begin{figure}
\centering
\begin{equation*}
\left[\begin{matrix}
0&0&1&0&0&0\\
0&1&-1&1&0&0\\
0&0&0&0&1&0\\
0&0&1&0&-1&1\\
1&0&0&0&0&0\\
0&0&0&0&1&0\\
\end{matrix}\right]
\leftrightarrow
\begin{matrix}
3&&&&\\
2&4&&&\\
2&4&5&&\\
2&3&4&6\\
1&2&3&4&6\\
1&2&3&4&5&6\\
\end{matrix} \leftrightarrow
\begin{matrix}
2&&&&\\
1&2&&&\\
1&2&2\\
1&1&1&2\\
0&0&0&0&1
\end{matrix}
\end{equation*}
\includegraphics[]{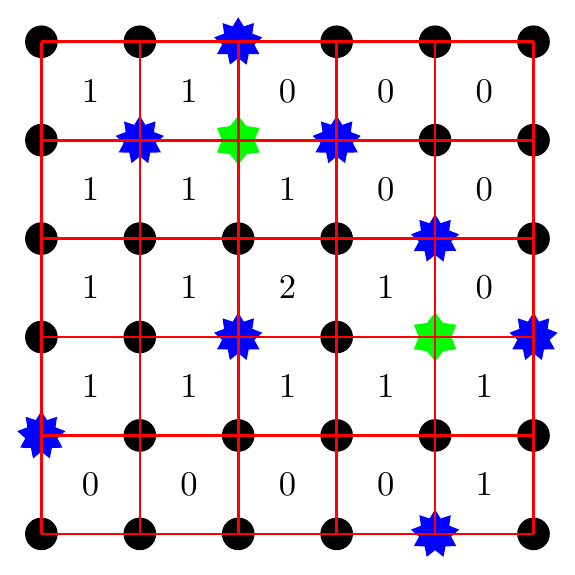}
\caption{An ASM and corresponding monotone triangle, reduced monotone triangle, and generalized Waldspurger matrix.  (The blue 9-sided stars represent 1's, and the green six-sided stars represent -1's in the transformation diagram.)}
\label{four in the same}
\end{figure}

\begin{figure}
\centering
\includegraphics[scale=1]{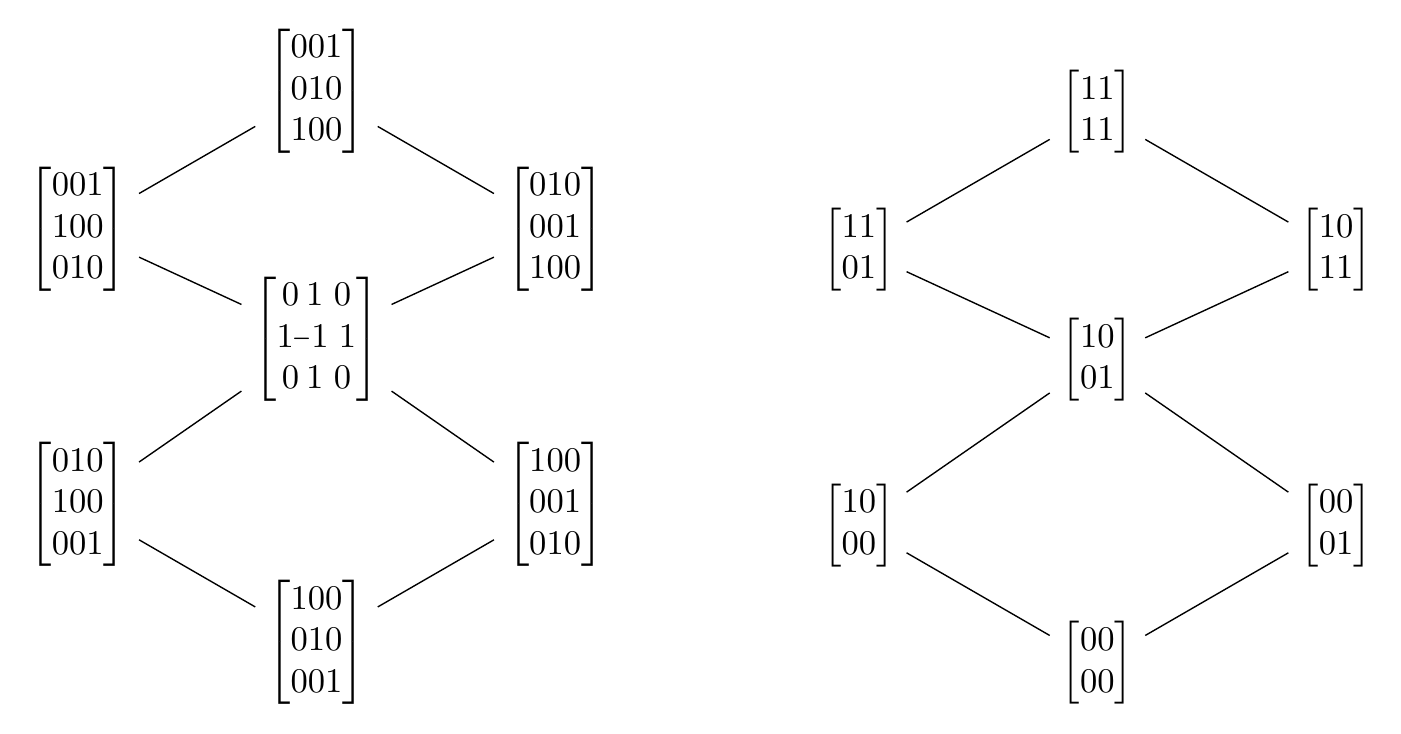}
\includegraphics[]{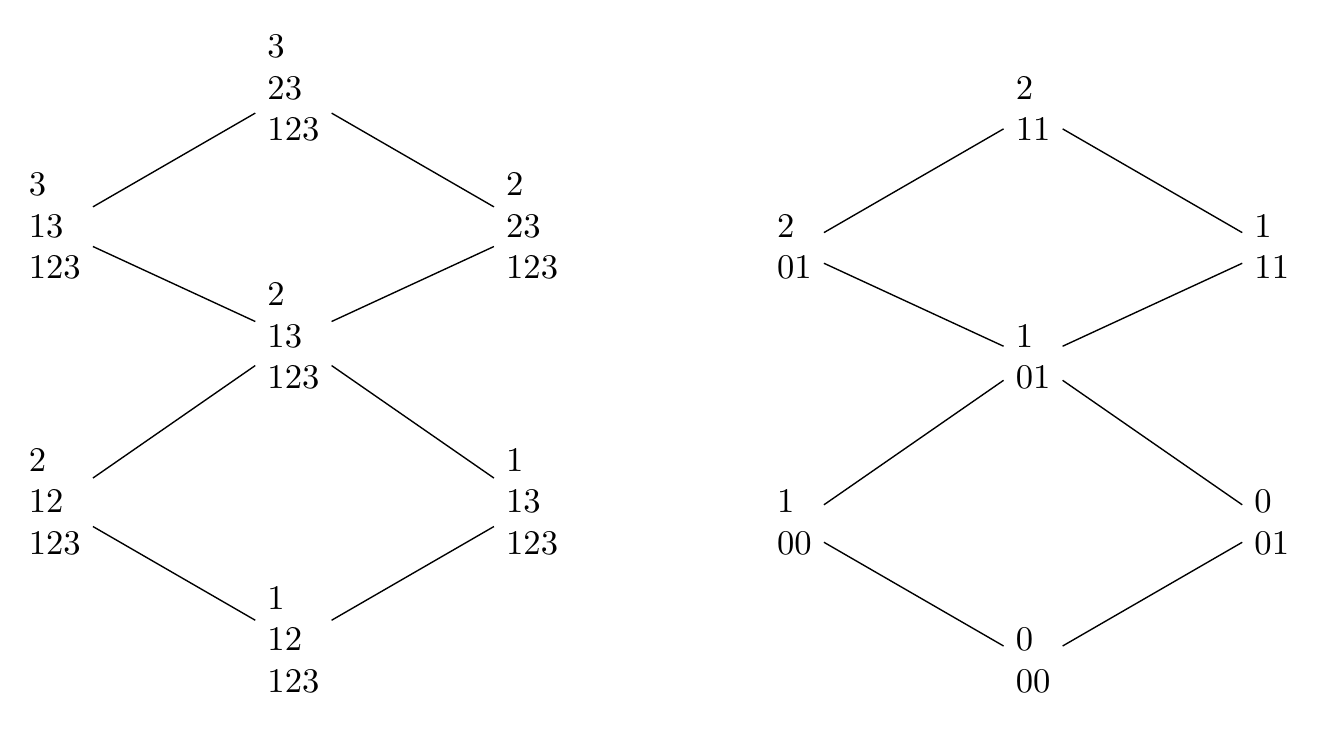}

\label{macnbruhat}
\caption{The Dedekind-MacNeille completion of Bruhat order $A_2$ viewed as ASMs, Generalized Waldspurger Matrices, monotone triangles, and reduced monotone triangles}
\end{figure}

\subsection{The Tetrahedral poset: Bigrassmannians and Join-Irreducibles}\label{tetrahedral section}
Lascoux and Sch{\"u}tzenberger showed that the Dedekind-MacNeille completion of Bruhat order for type A is a distributive lattice, and that its join-irreducible elements are exactly the bigrassmannian permutations \cite{sch} (those with a unique right descent and unique left descent). The number of bigrasmannian permutations of length $n$ is a \emph{tetrahedral number}, and is counted by the coefficients of $$\frac{1}{(1-z)^4}=1+4z+10z^2+20z^3+35z^4+\dots.$$
The relationship between the tetrahedral poset and ASMs has been studied elsewhere \cite{2014arXiv1408.5391S}, but our Waldspurger matrices provide a new prospective.  Bigrassmannian permutations correspond to Waldspurger matrices determined by fixing a single entry and then ``falling down" as quickly as possible.  More poetically, they are arrangements of oranges in a tetrahedral orange basket (held up so that one edge is parallel with the ground) so that only one orange may be removed without causing a tumble.  

\begin{figure}

\centering
$\begin{matrix}
1&1&1&1&1\\1&2&2&2&1\\1&2&3&2&1\\1&2&2&2&1\\1&1&1&1&1
\end{matrix}$
\caption{the number of ways of fixing each entry in a $5 \times 5$ Waldspurger matrix (also the top element in the Waldspurger version of the ASM lattice.)}
\label{tetrahedral}
\end{figure}

In our $A_n$ Waldspurger matrices, the number of ways of fixing a single entry to be a one is $n^2$
, to be a two is $(n-2)^2$, etc.  This sum of alternating squares is another well known formula for the tetrahedral numbers (see http://oeis.org/A000292 for more).
\subsection{Centers of Mass and Geometric Realizations of Hasse Diagrams}\label{centers of mass subsection}
Our definition for Waldspurger matrices was geometrically motivated, but we have seen that they are also very combinatorially related to the ASM lattice.  It is then natural to ask how this partial order and the geometry are connected.  One classical invariant of posets with a distinctly geometric flavor is the notion of order dimension.  The order dimension of a poset $P$ is the smallest $n$ for which $P \cong Q\subset \mathbb{R}^n$ where the elements of $Q$ are ordered componentwise.  In \cite{Reading2002}, Reading computed the order dimension of Bruhat orders for types A and B, the former being $\mathrm{dim}(A_n)=\lfloor \frac{(n+1)^2}{4}\rfloor$.  This tells us, in particular, that there is no way of embedding the lattice of $3 \times 3$ Waldspurger matrices in dimension less than $4$ in a way that preserves componentwise comparison.  On the other hand, for each of these $3 \times 3$ matrices, we have an associated simplex $\Delta_M \subset \mathbb{R}^3$ and may consider the natural map which takes $\Delta_M$ to its center of mass.

If one replaces each simplex $\Delta_{\pi}$ (where $\pi \in \mathfrak{S}_n$) with its center of mass, one gets back a translate of the vertex set of the classical permutohedron.  If one instead considers the centers of mass for each $\Delta(M)$ where $M$ is an alternating sign matrix, one obtains every interior point of the permutohedron as well; some appearing with multiplicities.  (see Figure \ref{embedded hasse a3}).  For example, the two generalized Waldspurger matrices below have the same center of mass.
$$\left[\begin{matrix}
1&0&0\\
0&1&0\\
0&1&1
\end{matrix}\right],\left[ \begin{matrix}
1&1&0\\
0&1&0\\
0&0&1
\end{matrix} \right].$$.

\begin{proposition}
The height statistic is not only the rank of an ASM $M$ in the lattice, it is also the height of the center of mass of $\Delta_M$ inside of the Meinrenken tile in the direction of $\rho$, the sum of the positive roots. 
\end{proposition}

\begin{proof}
We want to show that projection of the center of mass of $\Delta_M$ onto $\rho$ is (up to scalar multiple) equal to the sum of the entries in $\mathbf{WT}(M)$.
By the definition of $\Delta_M$, its center of mass is a scalar times the vector of column sums of $\mathbf{WT}(M)$. We will be done if we can show that projection of a vector $v$ onto $\rho$ is (up to scalar multiple) $\rho$ times the sum of the entries of $v$.

Projection of a vector $v$ onto $\rho$ in root coordinates, is $\frac{v^TC\rho}{\rho^TC \rho}\rho$.  The denominator is just a scalar, and the numerator is 
$$v^TC\rho=v^T\left[
\begin{matrix}
2&-1&0&\dots&\dots\\
-1&2&-1&0&\dots\\
0&-1&2&-1&\dots\\
\vdots&\vdots&\vdots&\vdots&\ddots\\
0&\dots&0&-1&2
\end{matrix}\right]
\left[\begin{matrix}
n\\
2(n-1)\\
3(n-2)\\
\vdots\\
(n-2)3\\
(n-1)2\\
n
\end{matrix}\right]=v^T\theta$$ where $\theta$ is the vector of all ones. We conclude that, up to scalars, this projection is the sum of the entries of $v$.
\end{proof}
\begin{figure}
\begin{center}
\includegraphics[scale=.9]{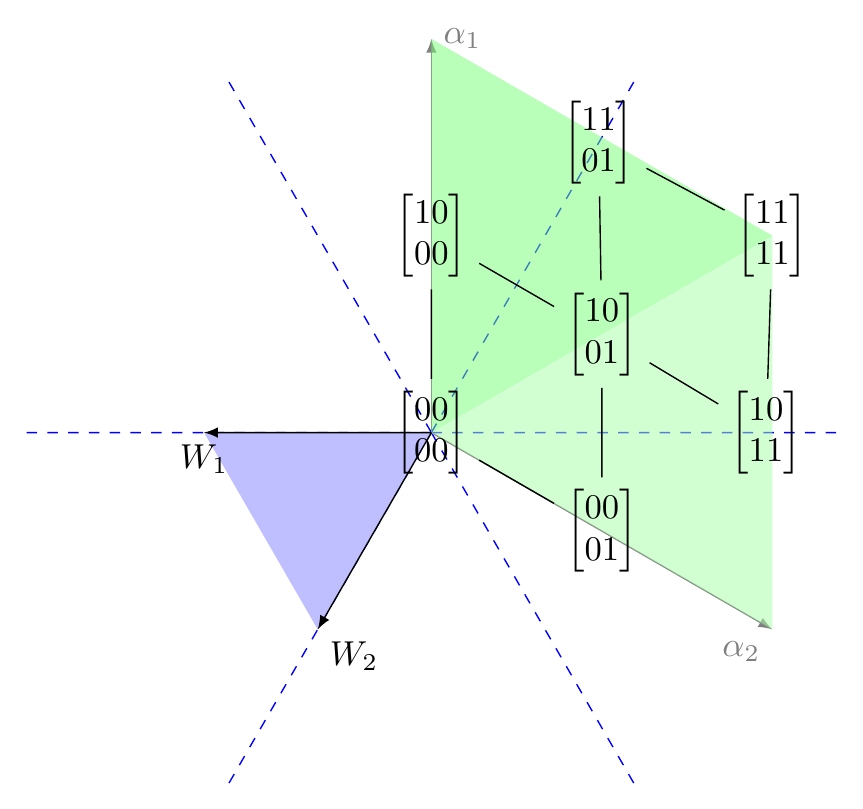}
\caption{Place $\mathbf{WT}(M)$ at the baricenter of $\Delta_M$ for each $M\in ASM$ to get a geometric realization of the Hasse diagram inside the Meinrenken tile}
\label{embedded hasse a2}
\end{center}
\end{figure}

\begin{figure}
\begin{center}
\includegraphics[scale=.75]{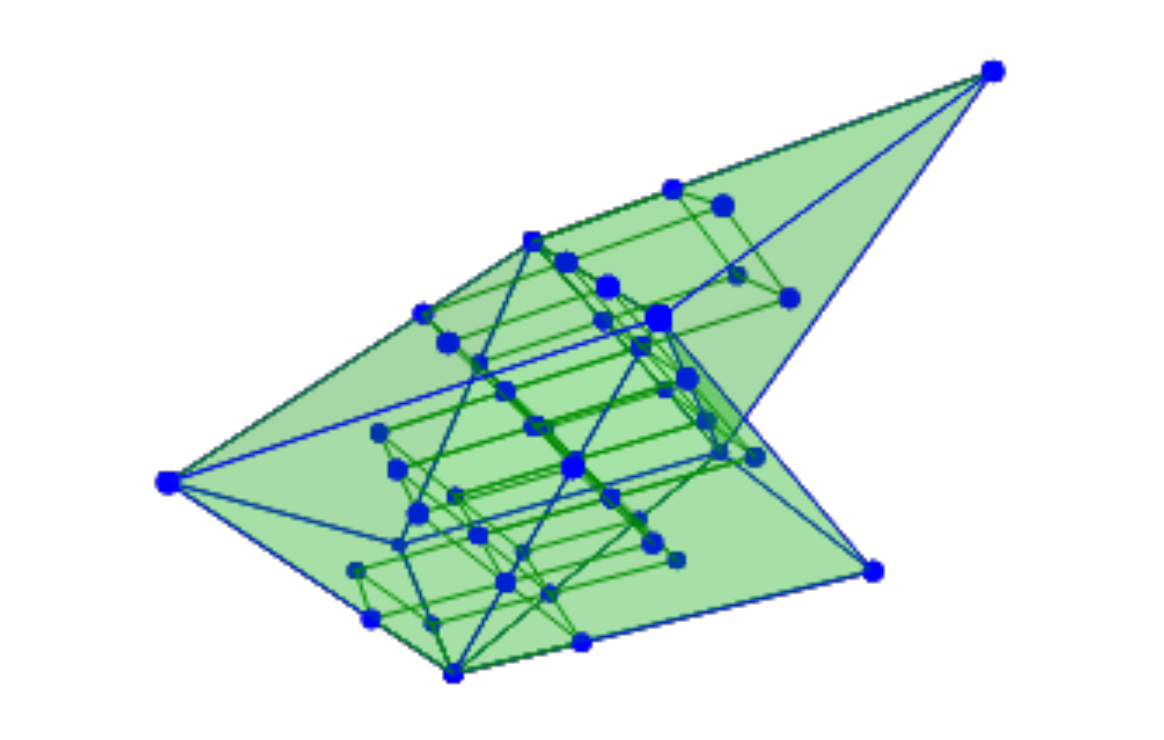}
\caption{There are 38 lattice points in the permutohedron and 42 Alternating sign matrices in this dimension.  Four of the interior points have multiplicity two.}
\label{embedded hasse a3}
\end{center}
\end{figure}


\newpage
\section{Types B and C}\label{types b&c section}
For general crystalographic root systems, $\Phi$, define the Waldspurger Transform of a Weyl group element $g$ to be the matrix
$$\mathbf{WT}_{\Phi}(g):=(Id-R_g)C_{\Phi}^{-1}$$
where $R_g$ is the matrix of $g$ in the coordinates of the simple roots of $\Phi$, and $C_{\Phi}$ is the Cartan Matrix.

If no root system is specified, we will assume type A, so that $\mathbf{WT}=\mathbf{WT}_{A}$ is the Waldspurger transform already discussed. Recall that we proved that 
$$\left[\mathbf{WT}(\pi)\right]_{i,j}= \begin{cases} 
      \displaystyle\sum_{\substack{a\leq i\\b>j}}\pi_{a,b} & i\leq j \\
      \displaystyle\sum_{\substack{a> i\\b\leq j}}\pi_{a,b} & i\geq j\\ 
   \end{cases}.
$$ 

\bigskip
It is natural to ask which phenomena we observed in type A will hold more generally.
It seems that the connection to Abelian ideals does not generalize.  One can verify that there are $2\cdot 3^{n-1}$ ``UM vectors of type B,'' but the author is unable to find any lie-theoretic interpretation of this fact.

The poset-theoretic results give more hope for generalization.  Lascoux and Sch{\"u}tzenberger showed that the Dedekind-MacNeille completion of Bruhat order for type B is a distributive lattice, and gave a description of the join-irreducible elements as a subset of the bigrassmannian elements \cite{sch}.  They showed that, while the number of bigrasmannian elements is counted by the coefficients of 
\begin{equation}\label{type b bigrass formula}\frac{1}{(1-z)^5}+\frac{1}{(1-z)^4}=1+6z+19z^2+45z^3+161z^4+\dots \end{equation}
the number of join-irreducibles or elements of the ``base'' are the \emph{octahedral numbers}:
\begin{equation}\label{type b base formula}\frac{(1+z)^2}{(1-z)^4}=1+6z+19z^2+44z^3+146z^4+\dots. \end{equation}
Geck and Kim \cite{geck1997bases} gave a more in-depth treatment of exactly when bigrassmannian elements fail to be part of the base, and Reading gave a combinatorial description of the base in terms of signed monotone triangles \cite{Reading2002}.  Recently, Anderson gave another combinatorial description of the base in terms of type B Rothe diagrams and essential sets \cite{2016arXiv161208670A}. Despite all this, the story is still a bit unsatisfying; there is no known combinatorial description for all of the elements of the Dedekind-MacNeille completion of Bruhat order for type B.  Reading's signed monotone triangles give us a means of determining whether or not a bigrassmannian element is in the base, but they are somehow not the right analog of ``type B alternating sign matrices''.  We encounter similar complications here, but we will nevertheless define type B and C Waldspurger matrices and push the theory as far as we can.

Analogous to our theorem for type A, we conjecture the following:
\begin{conjecture} Each element in the base for types B and C corresponds to a type B and C Waldspurger matrix which is componentwise least given a single fixed entry. \label{b bases conj}\end{conjecture}
We will work our way into the type B and C combinatorics in the following subsections.  We will then provide evidence supporting conjecture, while explaining where problems arise.

\subsection{Centrally Symmetric Permutation matrices}\label{Csym subsection}

For type B, we may consider $V=\mathbb{R}^n$, and $\Phi$ consisting of all integer vectors in $V$ of length $1$ or $\sqrt[]{2}$, for a total of $2n^2$ roots. Choose the simple roots: $\alpha_i=e_i-e_{i+1}$, for $1\leq i \leq n-1$ and the shorter root $\alpha_n = e_n$.  For type C, we will have the simple roots with the exception that $\alpha_n=2\cdot e_n$.

These root systems share a Weyl group of size $2^n n!$ consisting of the $n \times n$ {\it signed permutation matrices}.  That is, the set of all $n \times n$ permutation matrices where any of the $1$'s may be replaced with $-1$'s.  We will abuse notation and use $B_n$ for both the root system of type B, and for this particular representation of the Weyl group for both types B and C.  


Call a square $n \times n$ matrix {\it centrally symmetric} if it is preserved under $180^{\circ}$ rotation; that is if $M_{i,j}=M_{n-i,n-j}$ for all $1 \leq i,j \leq n-1$.  Let $\mathfrak{CS}_n$ denote the set of centrally symmetric $n \times n$ matrices.
\begin{proposition}
The group $\mathfrak{CS}_{2n}\subset\mathfrak{S}_{2n}$ is isomorphic to the group $B_n$ of signed permutations via a ``folding move''.
\end{proposition}
\begin{proof}
If $\pi$ is a $2n \times 2n$ centrally symmetric permutation matrix, we may ``fold'' it to obtain $\pi^{\star}$, a signed permutation on $n$, by letting $$\pi^{\star}_{i,j}=
\pi_{i,j}-\pi_{2n-i+1,j}$$
The map is invertible because $\pi$ was a permutation matrix, meaning that 
$$\pi^{\star}_{i,j}=\begin{cases}
1 & \textrm{ if }\pi_{i,j}=1\\
-1 & \textrm{ if }\pi_{2n-i+1,j}=1\\
0 \textrm{ otherwise}
\end{cases}$$
i.e. there will never be any collisions in the folding.
\end{proof}


\bigskip

We may also consider a similar ``folding map'' on the centrally symmetric type A Waldspurger matrices. 
$$\mathcal{F}:\mathbf{WT}_{A_{2n-1}}(\mathfrak{CS_{2n}})\longrightarrow \textrm{Mat}_n$$
Where $$\mathcal{F}(M)_{i,j}=\begin{cases}M_{i,j}+M_{2n-i+1,j}\textrm{ for all }1\leq i,j < n\\
M_{i,j} \textrm{ for all } i=n, j\leq n
\end{cases}$$

\begin{subsection}{Working in root coordinates}\label{b in rootcoords subsection}
Let $P$ be the change of basis matrix that gives the simple roots of $B_n$ in terms of the standard basis vectors, and define $Q$ analogously for $C_n$.  From the discussion above, we have:
$$P_{i,j}=\begin{cases}
1 & \textrm{ if }i=j\\
-1 & \textrm{ if }i=j+1\\
0 & \textrm{ otherwise}
\end{cases} \hspace{20pt}
Q_{i,j}=\begin{cases}
1 & \textrm{ if }i=j, j<n\\
-1 & \textrm{ if }i=j+1\\
2&\textrm{ if }i=j=n\\
0 & \textrm{ otherwise}
\end{cases}.$$
One can then verify that
$$
P_{i,j}^{-1}=\begin{cases}
1 & \textrm{ if }i\geq j\\
0 & \textrm{ otherwise}
\end{cases}\hspace{20pt}
Q_{i,j}^{-1}=\begin{cases}
1 & \textrm{ if }i=j, j<n\\
1/2& \textrm{ if }i=j=n\\
0 & \textrm{ otherwise}
\end{cases}.
$$
With respect to this ordering on the simple roots, one can further verify that the inverses of the Cartan matrices for the root systems $B_n$ and $C_n$ are, respectively:

$$(C_{B_{n}}^{-1})_{i,j}=\begin{cases}
\textrm{min}(i,j) & \textrm{ if } j<n\\
i/2 & \textrm{ if } j=n
\end{cases}\hspace{20pt}
(C_{C_{n}}^{-1})_{i,j}=\begin{cases}
\textrm{min}(i,j) & \textrm{ if } i<n\\
j/2 & \textrm{ if } i=n
\end{cases}.
$$
Next, if we let $S=Q(C_{C_{n}}^{-1})$, and let $R=P(C_{B_{n}}^{-1}),$ one may verify that
$$S=\begin{cases}
1 & \textrm{ if }j\geq i\\
0 & \textrm{ otherwise}
\end{cases}\hspace{20pt}
R=\begin{cases}
1&\textrm{ if }j\geq i, j\neq n\\
1/2&\textrm{ if }j=n\\
0&\textrm{ otherwise}
\end{cases}.
$$
\end{subsection}
\begin{theorem}
$\mathcal{F}$ is a bijection between centrally symmetric Waldspurger Matrices of type $A_{2n-1}$, and Waldspurger Matrices of type $C_n$.  
and the following diagram commutes:

\[ \begin{tikzcd}
B_n \arrow{r}{\mathbf{WT}_{C_n}}  & \mathbf{WT}_{C_n}(B_n)  \\%
\mathfrak{CS}_{2n} \arrow{r}{\mathbf{WT}}\arrow[swap]{u}{\star}& \mathbf{WT}(\mathfrak{CS_{2n}})\subset UM_{n-1}\arrow{u}{\mathcal{F}}
\end{tikzcd}
\]
\end{theorem}

\begin{proof}  We will show that $\mathcal{F}(\mathbf{WT}(\pi))_{i,j}$ and $\mathbf{WT}_{C_n}(\pi^{\star})_{i,j}$ are summing over the same parts of the permutation matrix $\pi$.  On the one hand,
\begin{gather}
\begin{align*}
\mathcal{F}(\mathbf{WT}(\pi))_{i,j} &=\begin{cases}\mathbf{WT}(\pi)_{i,j}+\mathbf{WT}(\pi)_{2n-i+1,j}\textrm{ for all }1\leq i,j < n\\
\mathbf{WT}(\pi)_{i,j} \textrm{ for all } i=n, j\leq n
\end{cases}\\
&=\begin{cases} 
      \displaystyle\sum_{\substack{a\leq i\\b>j}}^{2n}\pi_{a,b} +\displaystyle\sum_{\substack{a> 2n-i+1\\b\leq j}}^{2n}\pi_{a,b} & i\leq j<n \\
      \displaystyle\sum_{\substack{a> i\\b\leq j}}^{2n}\pi_{a,b} +\displaystyle\sum_{\substack{a> 2n-i+1\\b\leq j}}^{2n}\pi_{a,b} & j\leq i<n\\ 
      \displaystyle\sum_{\substack{a> i\\b\leq j}}^{2n}\pi_{a,b} & i=n
\end{cases}
\end{align*}.
\end{gather}
On the other hand,
\begin{align*}
\mathbf{WT}_{C_n}(\pi^{\star})_{i,j} &= \left(Id-(Q^{-1} \pi^{\star} Q)C_{C_{n}}^{-1}\right)_{i,j}\\
&=\left(C_{C_{n}}^{-1}-(Q^{-1} \pi^{\star} S)\right)_{i,j}\\
&=\left(C_{C_{n}}^{-1}\right)_{i,j}-\left((Q^{-1} \pi^{\star} S)\right)_{i,j}\\
&=\begin{cases}
\textrm{min}(i,j) & \textrm{ if } i<n\\
j/2 & \textrm{ if } i=n
\end{cases}
- \begin{cases}\displaystyle\sum_{a\leq i,b\leq j}\pi^{\star}_{a,b}& \textrm{if }i<n\\
\frac{1}{2}\displaystyle\sum_{a\leq i,b\leq j}\pi^{\star}_{a,b}& \textrm{if }i=n\end{cases}\\
&=\begin{cases}
\textrm{min}(i,j)-\displaystyle\sum_{a\leq i,b\leq j}^{2n}\pi_{a,b}-\pi_{2n-a+1,b} & \textrm{ if } i<n\\
\frac{j}{2} -\frac{1}{2}\displaystyle\sum_{a\leq i,b\leq j}^{2n}\pi_{a,b}-\pi_{2n-a+1,b} & \textrm{ if } i=n
\end{cases}\\
&=\begin{cases}
i-\displaystyle\sum_{a\leq i,b\leq j}^{2n}\pi_{a,b}-\pi_{2n-a+1,b} & \textrm{ if } i\leq j<n\\
j-\displaystyle\sum_{a\leq i,b\leq j}^{2n}\pi_{a,b}-\pi_{2n-a+1,b} & \textrm{ if } j\leq i<n\\
\frac{j}{2} -\frac{1}{2}\displaystyle\sum_{a\leq i,b\leq j}^{2n}\pi_{a,b}-\pi_{2n-a+1,b} & \textrm{ if } i=n
\end{cases}\\
&=\begin{cases}
\displaystyle\sum_{a\leq i}^{2n}\pi_{a,b}-\displaystyle\sum_{a\leq i,b\leq j}^{2n}\pi_{a,b}+\displaystyle\sum_{\substack{a> 2n-i+1\\b\leq j}}^{2n}\pi_{a,b} & \textrm{ if } i\leq j<n\\
\displaystyle\sum_{b\leq j}^{2n}\pi_{a,b}-\displaystyle\sum_{a\leq i,b\leq j}^{2n}\pi_{a,b}+\displaystyle\sum_{\substack{a> 2n-i+1\\b\leq j}}^{2n}\pi_{a,b} & \textrm{ if } j\leq i<n\\
\frac{j}{2} -\frac{1}{2}\displaystyle\sum_{a\leq i,b\leq j}^{2n}\pi_{a,b}-\pi_{2n-a+1,b} & \textrm{ if } i=n
\end{cases}\\
&=\begin{cases} 
      \displaystyle\sum_{\substack{a\leq i\\b>j}}^{2n}\pi_{a,b} +\displaystyle\sum_{\substack{a> 2n-i+1\\b\leq j}}^{2n}\pi_{a,b} & i\leq j<n \\
      \displaystyle\sum_{\substack{a> i\\b\leq j}}^{2n}\pi_{a,b} +\displaystyle\sum_{\substack{a> 2n-i+1\\b\leq j}}^{2n}\pi_{a,b} & j\leq i<n\\ 
      \displaystyle\sum_{\substack{a> i\\b\leq j}}^{2n}\pi_{a,b} & i=n
\end{cases}.
\end{align*}

The last equality is perhaps easier to see pictorially. The case $i\leq j<n$ says

\resizebox{.21\textwidth}{!}{$\begin{pmatrix} 
\ddots & \vdots & \multicolumn{1}{c|}{\vdots} & \vdots & \udots  \\ 
\dots & \pi_{i-1,j-1} &\multicolumn{1}{c|}{ \pi_{i-1,j}} & \textrm{sum these} & \dots \\ 
\dots & \pi_{i,j-1} & \multicolumn{1}{c|}{\pi_{i,j}} & \textrm{entries} & \dots \\ \cmidrule{4-5} 
\dots & \pi_{i+1,j-1} & \pi_{i+1,j} & \pi_{i+1,j+1} & \dots \\
\udots & \vdots     & \vdots    & \vdots      & \ddots \end{pmatrix}+$}
\resizebox{.21\textwidth}{!}{$\begin{pmatrix} 
\ddots & \vdots & \vdots & \vdots & \udots  \\ 
\dots & \dots &\pi_{2n-i+1,j-1} & \pi_{2n-i+1,j} & \dots \\ 
\cmidrule{1-3}\dots & \textrm{sum these} & \multicolumn{1}{c|}{\pi_{2n-i,j-1}} & \pi_{2n-i,j} & \dots \\
\dots & \textrm{entries} & \multicolumn{1}{c|}{\pi_{2n-i-1,j-1}} & \pi_{2n-i-1,j} & \dots \\
\udots & \vdots     & \multicolumn{1}{c|}{\vdots}    & \vdots      & \ddots \end{pmatrix}=$}
\resizebox{.21\textwidth}{!}{$\begin{pmatrix} 
\ddots & \vdots & \vdots & \vdots & \udots  \\ 
\dots & \textrm{sum these} & \pi_{i-1,j} & \pi_{i-1,j+1} & \dots \\ 
\dots & \textrm{entries} & \pi_{i,j} & \pi_{i,j+1} & \dots \\
\cmidrule{1-5}\dots & \pi_{i+1,j-1} & \pi_{i+1,j} & \pi_{i+1,j+1} & \dots \\
\udots & \vdots     & \vdots    & \vdots      & \ddots \end{pmatrix}-$}
\resizebox{.21\textwidth}{!}{$\begin{pmatrix} 
\ddots & \vdots & \multicolumn{1}{c|}{\vdots} & \vdots & \udots  \\ 
\dots & \textrm{sum these} &\multicolumn{1}{c|}{ \pi_{i-1,j}} & \pi_{i-1,j+1} & \dots \\ 
\dots & \textrm{entries} & \multicolumn{1}{c|}{\pi_{i,j}} & \pi_{i,j+1} & \dots \\
\cmidrule{1-3}\dots & \pi_{i+1,j-1} & \pi_{i+1,j} & \pi_{i+1,j+1} & \dots \\
\udots & \vdots     & \vdots    & \vdots      & \ddots \end{pmatrix}+$}
\resizebox{.21\textwidth}{!}{$\begin{pmatrix} 
\ddots & \vdots & \vdots & \vdots & \udots  \\ 
\dots & \dots &\pi_{2n-i+1,j-1} & \pi_{2n-i+1,j} & \dots \\ 
\cmidrule{1-3}\dots & \textrm{sum these} & \multicolumn{1}{c|}{\pi_{2n-i,j-1}} & \pi_{2n-i,j} & \dots \\
\dots & \textrm{entries} & \multicolumn{1}{c|}{\pi_{2n-i-1,j-1}} & \pi_{2n-i-1,j} & \dots \\
\udots & \vdots     & \multicolumn{1}{c|}{\vdots}    & \vdots      & \ddots \end{pmatrix}.$}

The $j\leq i<n$ case says

\resizebox{.21\textwidth}{!}{$\begin{pmatrix} 
\ddots & \vdots & \multicolumn{1}{c|}{\vdots} & \vdots & \udots  \\ 
\dots & \textrm{sum} &\multicolumn{1}{c|}{ \pi_{i-1,j}} & \pi_{i-1,j+1} & \dots \\ 
\dots & \textrm{these} & \multicolumn{1}{c|}{\pi_{i,j}} & \pi_{i,j+1} & \dots \\
\dots & \textrm{entries} & \multicolumn{1}{c|}{\pi_{i+1,j}} & \pi_{i+1,j+1} & \dots \\
\udots & \vdots     & \multicolumn{1}{c|}{\vdots}    & \vdots      & \ddots \end{pmatrix}$}
$-$\resizebox{.21\textwidth}{!}{$\begin{pmatrix} 
\ddots & \vdots & \multicolumn{1}{c|}{\vdots} & \vdots & \udots  \\ 
\dots & \textrm{sum these} &\multicolumn{1}{c|}{ \pi_{i-1,j}} & \pi_{i-1,j+1} & \dots \\ 
\dots & \textrm{entries} & \multicolumn{1}{c|}{\pi_{i,j}} & \pi_{i,j+1} & \dots \\
\cmidrule{1-3}\dots & \pi_{i+1,j-1} & \pi_{i+1,j} & \pi_{i+1,j+1} & \dots \\
\udots & \vdots     & \vdots    & \vdots      & \ddots \end{pmatrix}$}$+$
\resizebox{.21\textwidth}{!}{$\begin{pmatrix} 
\ddots & \vdots & \multicolumn{1}{c|}{\vdots} & \vdots & \udots  \\ 
\dots & \pi_{2n-i-1,j-1} &\multicolumn{1}{c|}{ \pi_{2n-i-1,j}} & \textrm{sum these} & \dots \\ 
\dots & \pi_{2n-i,j-1} & \multicolumn{1}{c|}{\pi_{2n-i,j}} & \textrm{entries} & \dots \\ \cmidrule{4-5} 
\dots & \pi_{2n-i+1,j-1} & \pi_{2n-i+1,j} & \pi_{2n-i+1,j+1} & \dots \\
\udots & \vdots     & \vdots    & \vdots      & \ddots \end{pmatrix}$}$=$
\resizebox{.21\textwidth}{!}{$\begin{pmatrix} 
\ddots & \vdots & \vdots & \vdots & \udots  \\ 
\dots & \pi_{i-1,j-1} &\pi_{i-1,j} & \pi_{i-1,j+1} & \dots \\ 
\dots & \pi_{i,j-1} & \pi_{i,j} & \pi_{i,j+1} & \dots \\
\cmidrule{1-3}\dots & \textrm{sum these entries} & \multicolumn{1}{c|}{\pi_{i+1,j}} & \pi_{i+1,j+1} & \dots \\
\udots & \vdots     & \multicolumn{1}{c|}{\vdots}    & \vdots      & \ddots \end{pmatrix}$}$+$
\resizebox{.21\textwidth}{!}{$\begin{pmatrix} 
\ddots & \vdots & \multicolumn{1}{c|}{\vdots} & \vdots & \udots  \\ 
\dots & \pi_{2n-i-1,j-1} &\multicolumn{1}{c|}{ \pi_{2n-i-1,j}} & \textrm{sum these} & \dots \\ 
\dots & \pi_{2n-i,j-1} & \multicolumn{1}{c|}{\pi_{2n-i,j}} & \textrm{entries} & \dots \\ \cmidrule{4-5} 
\dots & \pi_{2n-i+1,j-1} & \pi_{2n-i+1,j} & \pi_{2n-i+1,j+1} & \dots \\
\udots & \vdots     & \vdots    & \vdots      & \ddots \end{pmatrix}.$}

Finally, the case $i=n$ says

\resizebox{.26\textwidth}{!}{$\begin{pmatrix} 
\ddots & \vdots & \vdots & \vdots & \udots  \\ 
\dots & \pi_{i-1,j-1} &\pi_{i-1,j} & \pi_{i-1,j+1} & \dots \\ 
\cmidrule{1-3}\dots & \textrm{sum these} & \multicolumn{1}{c|}{\pi_{i,j}} & \pi_{i,j+1} & \dots \\
\dots & \textrm{entries} & \multicolumn{1}{c|}{\pi_{i+1,j}} & \pi_{i+1,j+1} & \dots \\
\udots & \vdots     & \multicolumn{1}{c|}{\vdots}    & \vdots      & \ddots \end{pmatrix}=$}
$\frac{1}{2}$
\resizebox{.26\textwidth}{!}{$\begin{pmatrix} 
\ddots & \vdots & \multicolumn{1}{c|}{\vdots} & \vdots & \udots  \\ 
\dots & \textrm{sum} &\multicolumn{1}{c|}{ \pi_{i-1,j}} & \pi_{i-1,j+1} & \dots \\ 
\dots & \textrm{these} & \multicolumn{1}{c|}{\pi_{i,j}} & \pi_{i,j+1} & \dots \\
\dots & \textrm{entries} & \multicolumn{1}{c|}{\pi_{i+1,j}} & \pi_{i+1,j+1} & \dots \\
\udots & \vdots     & \multicolumn{1}{c|}{\vdots}    & \vdots      & \ddots \end{pmatrix}$}
$-\frac{1}{2}$\resizebox{.26\textwidth}{!}{$\begin{pmatrix} 
\ddots & \vdots & \multicolumn{1}{c|}{\vdots} & \vdots & \udots  \\ 
\dots & \textrm{sum these} &\multicolumn{1}{c|}{ \pi_{i-1,j}} & \pi_{i-1,j+1} & \dots \\ 
\dots & \textrm{entries} & \multicolumn{1}{c|}{\pi_{i,j}} & \pi_{i,j+1} & \dots \\
\cmidrule{1-3}\dots & \pi_{i+1,j-1} & \pi_{i+1,j} & \pi_{i+1,j+1} & \dots \\
\udots & \vdots     & \vdots    & \vdots      & \ddots \end{pmatrix}$}
$+\frac{1}{2}$\resizebox{.26\textwidth}{!}{$\begin{pmatrix} 
\ddots & \vdots & \vdots & \vdots & \udots  \\ 
\dots & \pi_{i-1,j-1} &\pi_{i-1,j} & \pi_{i-1,j+1} & \dots \\ 
\cmidrule{1-3}\dots & \textrm{sum these} & \multicolumn{1}{c|}{\pi_{i,j}} & \pi_{i,j+1} & \dots \\
\dots & \textrm{entries} & \multicolumn{1}{c|}{\pi_{i+1,j}} & \pi_{i+1,j+1} & \dots \\
\udots & \vdots     & \multicolumn{1}{c|}{\vdots}    & \vdots      & \ddots \end{pmatrix}.$}

\end{proof}

Because transposition and the map $\star :\mathfrak{CS}_{2n}\rightarrow B_n$ commute, we will from now on abuse notation and identify centrally symmetric permutations with their images in $B_n=C_n$.
\begin{proposition} 
$\mathbf{WT}_{C_n}(\pi^{\intercal})=(\mathbf{WT}_{B_n}(\pi))^{\intercal}$ for any $\pi \in B_n$.
\end{proposition}
\begin{proof}
\begin{align*}
\mathbf{WT}_{C_n}(\pi^{\intercal}) &= Id-(Q^{-1} \pi^{\intercal} Q)C_{C_{n}}^{-1}\\
&=C_{C_{n}}^{-1}-(Q^{-1} \pi^{\intercal} S)\\
&=(C_{B_{n}}^{-1})^{\intercal}-(R^{\intercal} \pi^{\intercal} (P^{-1})^{\intercal})\\
&=Id-(P^{-1} \pi R)^{\intercal}(C_{B_{n}}^{-1})^{\intercal}\\
&=(Id-(P^{-1} \pi R)C_{B_{n}}^{-1})^{\intercal}\\
&=(\mathbf{WT}_{B_n}(\pi))^{\intercal}.
\end{align*}
\end{proof}

Informally, this proposition tells us that as far as the Waldspurger and Meinrenken theorems are concerned, types B and C are essentially the same.

\subsection{Smallest examples in full detail}\label{b2 deets subsection}

\begin{figure}
\centering 
\includegraphics[scale=.5]{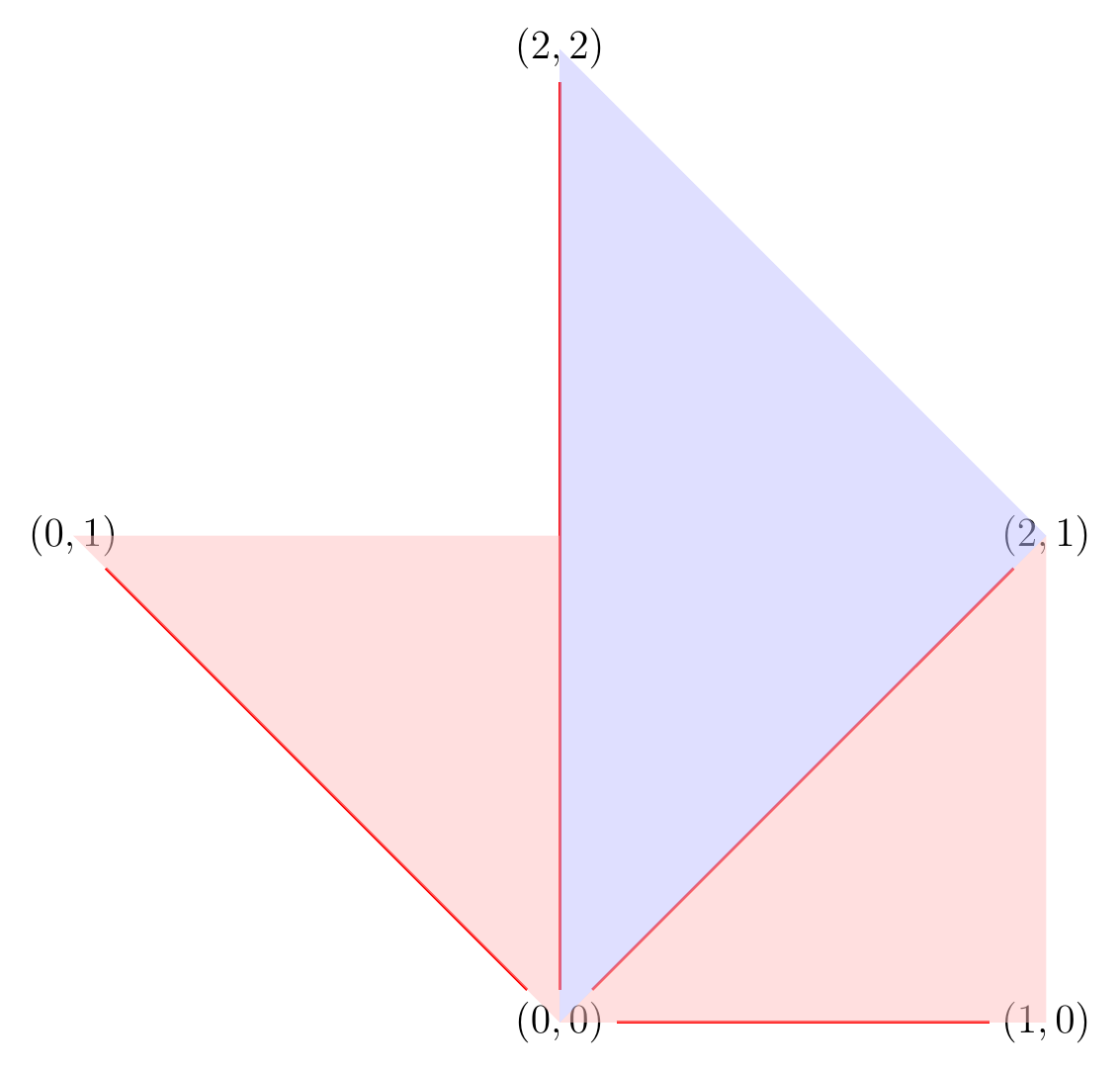}
\includegraphics[scale=.5]{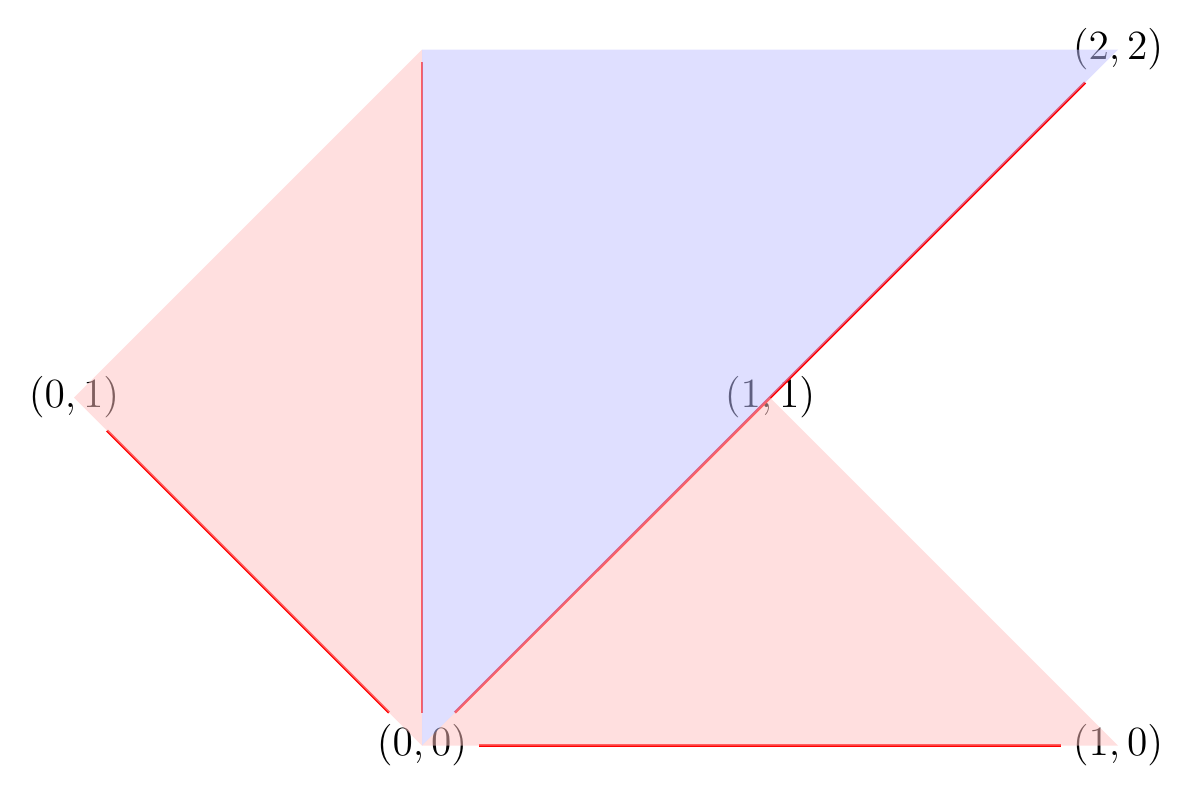}
\caption{The Meinrenken tiles for $C_2$ and $B_2$ respectively}
\end{figure}

There are exactly eight centrally symmetric $3 \times 3$ Waldspurger matrices of type A:
$$\left[\begin{matrix}
0&0&0\\
0&0&0\\
0&0&0
\end{matrix}\right], \hspace{10pt}
\left[\begin{matrix}
1&0&0\\
0&0&0\\
0&0&1
\end{matrix}\right], \hspace{10pt}
\left[\begin{matrix}
0&0&0\\
0&1&0\\
0&0&0
\end{matrix}\right], \hspace{10pt}
\left[\begin{matrix}
1&1&0\\
0&1&0\\
0&1&1
\end{matrix}\right] \textrm{ }$$ $$
\left[\begin{matrix}
1&0&0\\
1&1&1\\
0&0&1
\end{matrix}\right], \hspace{10pt}
\left[\begin{matrix}
1&1&1\\
1&1&1\\
1&1&1
\end{matrix}\right], \hspace{10pt}
\left[\begin{matrix}
1&1&0\\
1&2&1\\
0&1&1
\end{matrix}\right], \hspace{10pt}
\left[\begin{matrix}
1&1&1\\
1&2&1\\
1&1&1
\end{matrix}\right].$$

We may fold them vertically to get type C Waldspurger matrices, or horizontally to get type B:
$$\left[\begin{matrix}
0&0\\
0&0
\end{matrix}\right], \hspace{10pt}
\left[\begin{matrix}
1&0\\
0&0
\end{matrix}\right], \hspace{10pt}
\left[\begin{matrix}
0&0\\
0&1
\end{matrix}\right], \hspace{10pt}
\left[\begin{matrix}
1&2\\
0&1
\end{matrix}\right], \hspace{10pt}
\left[\begin{matrix}
1&0\\
1&1
\end{matrix}\right], \hspace{10pt}
\left[\begin{matrix}
2&2\\
1&1
\end{matrix}\right], \hspace{10pt}
\left[\begin{matrix}
1&2\\
1&2
\end{matrix}\right], \hspace{10pt}
\left[\begin{matrix}
2&2\\
1&2
\end{matrix}\right]
$$
$$\left[\begin{matrix}
0&0\\
0&0
\end{matrix}\right], \hspace{10pt}
\left[\begin{matrix}
1&0\\
0&0
\end{matrix}\right], \hspace{10pt}
\left[\begin{matrix}
0&0\\
0&1
\end{matrix}\right], \hspace{10pt}
\left[\begin{matrix}
1&1\\
0&1
\end{matrix}\right], \hspace{10pt}
\left[\begin{matrix}
1&0\\
2&1
\end{matrix}\right], \hspace{10pt}
\left[\begin{matrix}
2&1\\
2&1
\end{matrix}\right], \hspace{10pt}
\left[\begin{matrix}
1&1\\
2&2
\end{matrix}\right], \hspace{10pt}
\left[\begin{matrix}
2&1\\
2&2
\end{matrix}\right].
$$
Recall that, in type A, the dimensions of each of the simplices was determined by the number of cycles of the corresponding permutation, and so the number of simplices of a given dimension was a Stirling number of the first kind.  In type B, we see Suter's type B stirling numbers of the first kind \cite{stirlingB} with our $1$ point, $4$ edges, and $3$ triangles for $B_2$ and $C_2$.  
In this dimension there are two centrally symmetric ASMS which are not permutations, with type A Waldspurger matrices 
$\left [\begin{matrix}
1&0&0\\0&1&0\\0&0&1
\end{matrix}\right]$
and
$\left [\begin{matrix}
1&1&0\\1&1&1\\0&1&1
\end{matrix}\right ]$
.  They fold vertically to give us the two extra matrices picture in the right hand side of Figure \ref{b2comp}.
\begin{figure}
\begin{center}
\includegraphics[]{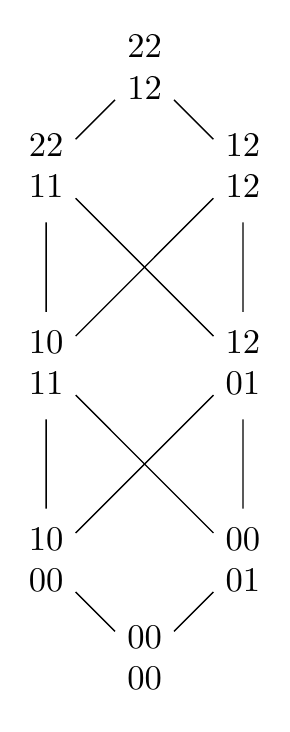}
\hspace{30pt}
\includegraphics[]{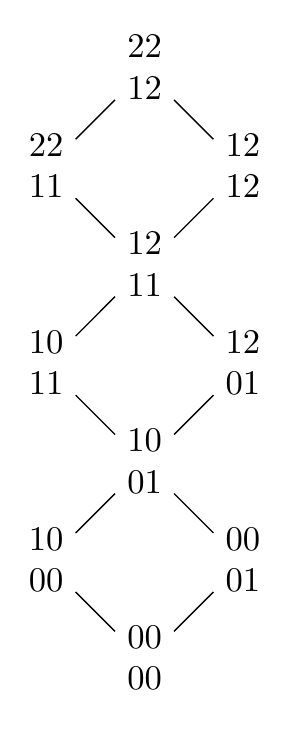}
\caption{In the case of $C_2$, (and $B_2$, though it is not shown here) componentwise comparison of Waldspurger matrices is exactly Bruhat order, and componentwise comparison of folded centrally symmetric ASMs is exactly its Dedekind-MacNeille completion.}
\label{b2comp}
\end{center}
\end{figure}

\subsection{UM vectors for types B and C}
While the our folding map $\mathcal{F}$ did not double the middle row of type A centrally symmetric Waldspurger matrices (folding it onto itself), it is combinatorially convenient for us to do so.  We will call such a map $\tilde{\mathcal{F}}$.  The following proposition combines the inequality description of UM vectors from Theorem \ref{inequality} and $\tilde{\mathcal{F}}$ to give inequality descriptions for ``UM vectors for types B and C''.

\begin{proposition}
Column and row vectors of type B and C Waldspurger matrices (with respect to $\tilde{\mathcal{F}}$) must start with entries $0,1,2$, increase by $0,1,$ or $ 2$ up to the diagonal, and increase by $-1,0,$ or $1$ after the diagonal, ending with an even number. \end{proposition} 

\begin{corollary}
There are $2\cdot 3^{n-1}$ UM vectors of type B.
\end{corollary}
\subsection{Conjectural description of elements in the Base}\label{base conjecture subsection}
Recall that, in type A, one could obtain any element of the base by specifying a single entry in a Waldspurger matrix (as long as it was below the corresponding entry in the Waldspurger matrix corresponding to the longest word) and ``falling down'' (see Figure \ref{tetrahedral}).  In contrast, the type C Waldspurger matrix for the longest word (with respect to $\tilde{\mathcal{F}}$) is 
$$\left[\begin{matrix}
2&2&2&2&\dots &2&2\\
2&4&4&4&\dots &4&4\\
2&4&6&6&\dots &6&6\\
\dots&\dots&\dots&\dots& \dots & \dots\\
2&4&6&8&\dots&2(n-1)&2(n-1)\\
2&4&6&8&\dots&2(n-1)&2n
\end{matrix}\right]
$$
but when constructing a Type C Waldspurger matrix, we are no longer free to specify entries on the right or bottom boundaries to be odd.  Consequentially, the type C analog of Figure \ref{tetrahedral} as the number of ways of determining the $(i,j)$th entry of a type C Waldspurger matrices is

$$\begin{matrix}
2&2&2&2&\dots &2&1\\
2&4&4&4&\dots &4&2\\
2&4&6&6&\dots &6&3\\
\dots&\dots&\dots&\dots& \dots & \dots\\
2&4&6&8&\dots&2(n-1)&n-1\\
1&2&3&4&\dots&n-1&n
\end{matrix}.
$$
It is straight forward to verify that the entries above sum to an octahedral number, and this supports our Conjecture \ref{b bases conj}.

In $C_4$ we first see the distinction between bigrassmannian elements, and elements of the base.  Recall from Equations \ref{type b bigrass formula} and \ref{type b base formula} that there are $45$ bigrassmannian elements, but only $44$ elements in the base. The Waldspurger matrix of every bigrassmannian element is minimal with respect to a fixed single entry, but there is one collision.  There are two incomparable bigrassmannian elements that are minimal after fixing the $(2,2)$ entry to be a two:
$$\left[ 
\begin{matrix}
1&1&0&0\\
1&\colorbox{green}{$2$}&1&0\\
0&1&1&0\\
0&0&0&0
\end{matrix}
\right] \textrm{ vs }
\left[ 
\begin{matrix}
0&0&0&0\\
0&\colorbox{green}{$2$}&2&1\\
0&2&2&1\\
0&2&2&1
\end{matrix}
\right].$$
The matrix on the left is in the base, and the one on the right is not.  The question ``bigrassmannian vs base'' seems intimately connected to the question:  What is the type C analog of the type A ``falling down'' algorithm? 
These two matrices may be ``unfolded'' to the centrally symmetric type A Waldspurger matrices:
$$\left[\begin{matrix}
1&1&0&0&0&0&0\\
1&\colorbox{green}{$2$}&1&0&0&\colorbox{green}{$0$}&0\\
0&1&1&0&0&0&0\\
0&0&0&0&0&0&0\\
0&0&0&0&1&1&0\\
0&\colorbox{green}{$0$}&0&0&1&\colorbox{green}{$2$}&1\\
0&0&0&0&0&1&1
\end{matrix}
\right] \textrm{ and }
\left[\begin{matrix}
0&0&0&0&0&0&0\\
0&\colorbox{green}{$1$}&1&1&1&\colorbox{green}{$1$}&0\\
0&1&1&1&1&1&0\\
0&1&1&1&1&1&0\\
0&1&1&1&1&1&0\\
0&\colorbox{green}{$1$}&1&1&1&\colorbox{green}{$1$}&0\\
0&0&0&0&0&0&0
\end{matrix}
\right].$$
If we write $2=2+0$ and we ``fall down'' in the type A way, we get the matrix on the left.  If we write $2=1+1$ and ``fall down'' in the type A way, we get the matrix on the right.  Conjecturally, the Waldspurger matrices of bigrassmannian elements for $C_n$ are determined by specifying a single entry, unfolding it to specify four entries of a $(2n-1)\times(2n-1)$ centrally symmetric type A Waldspurger matrix, and then performing the type A falling down algorithm.  Elements of the base come from unfolding the specified entry as inequitably as possible.

\subsection{Waldspurger Order}
Define the \emph{Waldspurger Order} on a finite reflection group to be the componentwise order on Waldspurger matrices.  One is given hope in the low dimensions that Bruhat order might, as in type A, be merely componentwise comparison of Waldspurger Matrices.  This is true for $C_2$ and $C_3$ and in both cases, the Dedekind-MacNeille completion comes from simply folding centrally symmetric ASMs.  It fails for $C_n$ when $n\geq 4$, (though it appears to be an order extension).

Among the bigrassmannian elements of $C_4$, there are exactly two cover relations in Waldspurger order which are not cover relations in Bruhat order:
$$\left[ 
\begin{matrix}
1&1&0&0\\
1&2&1&0\\
0&1&1&0\\
0&0&0&0
\end{matrix}
\right]<
\left[ 
\begin{matrix}
2&2&2&1\\
2&2&2&1\\
2&2&2&1\\
2&2&2&1
\end{matrix}
\right] \textrm{ and }
\left[
\begin{matrix}
0&0&0&0\\
0&2&2&1\\
0&2&2&1\\
0&2&2&1
\end{matrix}
\right]<
\left[
\begin{matrix}
1&1&1&0\\
1&2&3&1\\
1&3&5&2\\
0&2&4&2
\end{matrix}
\right].
$$
Folded $8 \times 8$ centrally symmetric ASMs also fail to be a lattice with respect to componentwise order.  The same sort of failures were recognized for signed monotone triangles by Reading in Section 10, Question 4 of \cite{Reading2002}.

\section{Further Questions}\label{further questions section}
\begin{enumerate}
\item It is curious that the same elements which caused bigrassmannian$\neq$base, for $B_4$ and $C_4$ are involved in causing Waldspurger order$\neq$Bruhat order.  Is this a coincidence, or can one use it to give a concrete combinatorial description of elements of Dedekind-MacNeille completion of Bruhat order for types B and C using Waldspurger matrices?
\item Is there a simple way to determine a signed permutation's essential set in the sense of \cite{2016arXiv161208670A} from its Waldspurger matrix?
\item Is there a description of Waldspurger order in terms of words in the Coxeter group? 
\item How many elements are there in the Dedekind-MacNeille completion of Bruhat order for type B?
\item In \cite{meinrenken2009tilings}, Meinrenken has another intriguing theorem:  Let $W$ an affine Weyl group with $A$ a fundamental alcove. Then for any endomorphism $S$, in the same connected component as $0$ in the set $\left\{ S \in \textrm{End}(V) \mid \textrm{det}(S-w)\neq 0 \forall w \in W\right\}$, the simplices $(S-w)A$ for $w \in W$ are all disjoint and their closures cover the entire vector space $V$. \newline This theorem seems to provide an interesting interpolation between the affine hyperplane arrangement, or Stiefel diagram, and the Meinrenken tile.  Does any nice combinatorics arise from selecting nice endomorphisms?  Is there an intrinsic characterization of the types of tilings that arise in this way?
\item It is a classical result in Ehrhart theory \cite{postnikov2009permutohedra} that the number of points inside of the permutohedron is the number of forests on the vertex set $\{1,2,\dots n\}$. We have exhibited a surjective map from ASMs to these same points.  Is there an interpretation of these multiplicities in terms of forest structures?  
\end{enumerate}
\bibliography{sample}
\bibliographystyle{alpha}

\end{document}